\documentclass{amsart}
\usepackage[english]{babel}
\usepackage[latin1]{inputenc}
\usepackage{amsmath,amsfonts,amssymb,amsthm,amscd,array,stmaryrd,mathrsfs, mathdots, epigraph,euscript,longtable}
 \usepackage[all]{xy}
 \usepackage{url}
\usepackage{multirow, blkarray}
\usepackage{booktabs}
\usepackage{tikz}
\usepackage{textcomp}
 \usepackage[final]{epsfig}
 \usepackage{color}
\theoremstyle{plain}
\newtheorem{thm}{Theorem}[subsection]
\newtheorem{pb}{Problem}
\newtheorem{lem}[thm]{Lemma}
\newtheorem{cor}[thm]{Corollary}
\newtheorem{prop}[thm]{Proposition}

\theoremstyle{definition}
\newtheorem{defn}[thm]{Definition}
\newtheorem{rem}[thm]{Remark}
\newtheorem{ex}[thm]{Example}

\let\ssection=\section
\renewcommand{\section}{\setcounter{equation}{0}\ssection}

\newcommand{\R}{\mathbb{R}}
\newcommand{\Z}{\mathbb{Z}}
\newcommand{\C}{\mathbb{C}}

\newcommand{\pP}{{\mathbb{P}}}

\newcommand{\E}{\mathcal{E}}

\newcommand{\cL}{\mathcal{L}}

\newcommand{\Rc}{\mathcal{R}}

\newcommand{\gM}{\mathfrak{M}}
\newcommand{\gT}{\mathfrak{T}}

\newcommand{\id}{\textup{Id}}

\newcommand{\Id}{\mathrm{Id}}
\newcommand{\SL}{\mathrm{SL}}

\newcommand{\OSp}{\mathrm{OSp}}

\def\a{\alpha}
\def\b{\beta}
\def\d{\delta}

\def\g{\gamma}
\def\L{\Lambda}
\def\om{\omega}

\def\t{\tau}

\def\l{\lambda}
\def\m{\mu}

\begin{document}

\title[Supersymmetric frieze patterns and linear difference operators]{Introducing supersymmetric frieze patterns\\ and linear difference operators}

\author[S. Morier-Genoud, V. Ovsienko, S. Tabachnikov]{Sophie Morier-Genoud, Valentin Ovsienko and Serge Tabachnikov
\\
{\tiny with an Appendix by}
Alexey Ustinov}

\address{Sophie Morier-Genoud,
Sorbonne Universit\'es, UPMC Univ Paris 06, UMR 7586, 
Institut de Math\'ematiques de Jussieu- Paris Rive Gauche, 
Case 247, 4 place Jussieu, F-75005, Paris, France
}

\address{
Valentin Ovsienko,
CNRS,
Laboratoire de Math\'ematiques 
U.F.R. Sciences Exactes et Naturelles 
Moulin de la Housse - BP 1039 
51687 REIMS cedex 2,
France}

\address{
Serge Tabachnikov,
Pennsylvania State University,
Department of Mathematics,
University Park, PA 16802, USA,
and 
ICERM, Brown University, Box1995,
Providence, RI 02912, USA
}

\address{
Alexey Ustinov,
Institute of Applied Mathematics,
Khabarovsk Division, Russian Academy of Sciences,
54 Dzerzhinsky Street, Khabarovsk, 680000, Russia}

\email{sophie.morier-genoud@imj-prg.fr,
valentin.ovsienko@univ-reims.fr,
tabachni@math.psu.edu,
ustinov@iam.khv.ru
}



\begin{abstract}
We introduce a supersymmetric analog of the classical Coxeter frieze patterns.
Our approach is based on the relation with linear difference operators.
We define supersymmetric analogs of linear difference operators called Hill's operators.
The space of these ``superfriezes'' is an algebraic supervariety,
isomorphic to the space of supersymmetric second order difference equations,
called Hill's equations.
\end{abstract}

\maketitle

\tableofcontents

\section*{Introduction}

Frieze patterns were introduced by Coxeter~\cite{Cox},
and studied by Coxeter and Conway~\cite{CoCo}.
Frieze patterns are closely related to
classical notions of number theory,
such as continued fractions, Farey series, as well as the Catalan numbers.
Recently friezes have attracted much interest,
mainly because of their deep relation to the theory of cluster algebras
developed by Fomin and Zelevinsky, see~\cite{FZ1}--\cite{FZ4}.
This relation was pointed out in~\cite{FZ2} and developed in~\cite{CaCh}.
Generalized frieze patterns were defined in~\cite{ARS}.
Further relations to moduli spaces of configurations of points
in projective spaces and linear difference equations were studied in~\cite{MGOT,SVRS}.

The main goal of this paper is to study superanalogs of Coxeter's frieze patterns.
We believe that ``superfriezes'' introduced in this paper provide us with a first example of
cluster superalgebra.
We hope to investigate this notion in a more general setting elsewhere.

Our approach to friezes uses the connection with linear difference equations.
The discrete Sturm-Liouville (one-dimensional Schr\"odinger) equation is a second order
equation of the form:
$$
V_i=a_iV_{i-1}-V_{i-2},
$$
where the sequence $(V_i)$ is unknown, and where
the potential (or coefficient) $(a_i)$ is a given sequence.
Importance of linear difference equations is due to the fact that
many classical sequences of 
numbers, orthogonal polynomials, special functions, etc.
satisfy such equations.
Linear difference equations with periodic coefficients, i.e., $a_{i+n}=a_i$,
were recently used to study
discrete integrable systems related to cluster algebras, see \cite{OST,MGOT,SVRS,Kr}.
It turns out that one particular case, where
all the solutions of the above equation are
antiperiodic: 
$$
V_{i+n}=-V_i
$$ 
are of a special interest.
We call this special class of discrete Sturm-Liouville equations Hill's equations.
They form an algebraic variety which is isomorphic to the space
of Coxeter's friezes.

We understand a frieze pattern as just another way to represent
the corresponding Hill equation.
Roughly speaking, a frieze is a way to write potential and solutions
of a difference equation in the same infinite matrix.
Friezes provide a very natural coordinate system of the space
of Hill's equations that defines a structure of cluster manifold.

To the best of our knowledge, supersymmetric analogs of
difference equations have never been studied.
We introduce a class of supersymmetric difference equations
that are analogous to Hill's (or Sturm-Liouville, one-dimensional Schr\"odinger) equations.
We show that these difference equations can be identified with
superfriezes. 
The main ingredient of difference equations we consider is the
{\it shift operator} acting on sequences.
In the classical case, the shift operator is the linear operator $T$ 
defined by $(TV)_i=V_{i-1}$, discretizing the translation vector field $\frac{d}{dx}$.
We define a supersymmetric version of $T$,
as a linear operator $\gT$ acting on pairs of sequences
and satisfying $\gT^2=-T$.
This operator is a discretization of the famous odd supersymmetric vector field
$D=\partial_\xi-\xi\partial_x$.
The corresponding ``superfriezes'' are constructed with the help of
modified Coxeter's frieze rule where $\SL_2$ is replaced with
the supergroup $\OSp(1|2)$.

Discrete Sturm-Liouville equations with periodic potential
can be understood as discretization of the Virasoro algebra.
Two different superanalogs of the Virasoro algebra are known as
Neveu-Schwarz and Ramond algebras.
The first one is defined on the supercircle $S^{1|1}$ related to
the trivial $1$-dimensional bundle over $S^1$, while the second one
is associated with the twisted supercircle $S_+^{1|1}$
related to the M\"obius bundle.
The supersymmetric version we consider is the M\"obius (or Ramond) one.

The main results of the paper are Theorems~\ref{Glade} and~\ref{ISOMTH}.
The first theorem describes the main properties of superfriezes that are
very similar to those of Coxeter's friezes.
The second theorem identifies the spaces of superfriezes and Hill's equations.

The paper consists in three main sections.

In Section~\ref{LDOS}, we consider
supersymmetric difference operators.
The space of such operators with (anti)periodic solutions that we call Hill's operators
is an algebraic supervariety.

In Section~\ref{SFS}, we introduce analogs of Coxeter's friezes
in the supercase.
We establish the glide symmetry and periodicity of generic superfriezes.
We prove that the space of superfriezes is an algebraic supervariety
isomorphic to that of Hill's operators.
We give a simple direct proof of the Laurent phenomenon 
occurring in superfriezes.

Each of these main section includes a short introduction outlining the main features
of the respective classical theory.

Finally, in Section~\ref{OPSEc}, we formulate and discuss some of the open problems.

The space of Coxeter's frieze patterns is a cluster variety
associated to a Dynkin graph of type~$A$, see~\cite{CaCh}.
Frieze patterns can be taken as the basic class of cluster algebras
which explains the exchange relations and the mutation rules.

\section{Supersymmetric linear difference operators}\label{LDOS}

In this section, we introduce supersymmetric linear finite difference
operators and the corresponding linear finite difference equations,
generalizing the classical difference operators and
difference equations.

The difference operators are defined using the supersymmetric shift operator.

\subsection{Classical discrete Sturm-Liouville and Hill's operators}\label{ClSec}

We start with a brief reminder of well-known 
second order operators.

The Sturm-Liouville operator 
(also known as discrete one-dimensional Schr\"odinger operators or Hill's operators)
is a linear differential operator 
$$
\left(\frac{d}{dx}\right)^2+u(x)
$$
acting on functions in one variable.

The discrete version of the Sturm-Liouville operator
is the following linear operator
$$
L=T^2-a\,T+\id,
$$
acting on infinite sequences $V=(V_i)$,
where $i\in\Z$ and $T$ is the {\it shift operator} 
$$
(TV)_i=V_{i-1},
$$
and where $a=(a_i)$ is a given infinite sequence
called the {\it coefficient}, or {\it potential} of the operator.
The coefficient generates a diagonal operator, i.e., $(aV)_i=a_iV_i$.
The sequence $(a_i)$ is usually taken with values in $\R$, or $\C$.

Given an operator $L$, one can define
the corresponding linear recurrence equation $L(V)=0$,
that reads:
\begin{equation}
\label{SLEq}
V_i=a_iV_{i-1}-V_{i-2},
\end{equation}
for all $i\in\Z$.
The sequence $(V_i)$ is a {\it variable}, or {\it solution} of the equation.

The spectral theory of linear difference operators was extensively studied; 
see~\cite{Kr,KN} and references therein.
The importance of second order operators and equations is due to the fact
that many sequences of numbers and special functions satisfy such equations.

We will impose the following two conditions: 

\begin{enumerate}
\item[(a)]
the potential of the operator is $n$-periodic, i.e.,
$a_{i+n}=a_i$;
\item[(b)]
{\it all} solutions of the equation~(\ref{SLEq}) are $n$-antiperiodic:
$$
V_{i+n}=-V_i.
$$
\end{enumerate}

The condition (a) implies the existence of a {\it monodromy operator} $M\in\SL_2$
(defined up to conjugation).
The space of solutions of the equation~(\ref{SLEq}) is 2-dimensional;
the monodromy operator is defined as the action of the operator of shift by the period, 
$T^n$, to the space of solutions.
The condition (b) means that the monodromy operator
is:
$$
M=
\left(
\begin{array}{rr}
-1&0\\[4pt]
0&-1
\end{array}
\right).
$$
Note that condition (b) implies condition (a),
since the coefficients can be recovered from the solutions.
Any Sturm-Liouville operator satisfying conditions (a) and (b) will be called
{\it Hill's operator}.

The following statement is almost obvious; for a more general discussion, see~\cite{SVRS}.

\begin{prop}
\label{ClSLProp}
The space of Hill's operators 
is an algebraic variety of dimension $n-3$.
\end{prop}

\noindent
Indeed, the coefficients of the monodromy operator are polynomials in $a_i$'s,
and the condition $M=-\Id$ implies that the codimension is $3$.

It turns out that the algebraic variety of Hill's operators has a geometric meaning. 
Consider the {\it moduli space} of configurations of $n$ points
in the projective line, i.e., the space of $n$-tuples of points:
$$
\{v_1,\ldots,v_n\}\subset\pP^1,
\qquad
v_{i+1}\not=v_i
$$
modulo the action of $\SL_2$.
We will use the cyclic order, so that $v_{i+n}=v_i$, for all $i\in\Z$.
This space will be denoted by $\widehat{\mathcal{M}}_{0,n}$.
Note that this space is 
slightly bigger than the classical space~$\mathcal{M}_{0,n}$,
which is the configuration moduli space of $n$ {\it distinct} points in $\pP^1$.
The following statement is a particular case of a theorem proved in~~\cite{SVRS}.

\begin{thm}
\label{ClSLPropDva}
If $n$ is odd, then the algebraic variety  of Hill's operators is isomorphic to~$\widehat{\mathcal{M}}_{0,n}$.
\end{thm}

\noindent
For a detailed proof of this statement,
see~\cite{SVRS}.
The idea is as follows.
Given Hill's operator, 
choose arbitrary basis of two linearly independent solutions, $V^{(1)}$ and $V^{(2)}$.
One then defines a configuration of $n$ points in~$\pP^1$
taking for every $i\in\Z$ 
$$
v_i=(V_i^{(1)}:V_i^{(2)}).
$$
This $n$-tuple of points is defined modulo linear-fractional transformations
(homographies)
corresponding to the choice of the basis of solutions.

\subsection{Supersymmetric shift operator}
Let $\Rc=\Rc_{\bar0}\oplus\Rc_{\bar1}$ be an arbitrary supercommutative ring,
and $\widehat \Rc=\Rc\oplus \xi\Rc$ its extension, where $\xi$ is an odd variable.

We will be considering infinite sequences
$$
V+\xi{}W:=(V_i+\xi{}W_i),
\qquad{}i\in\Z,
$$
where $V_i,W_i\in\Rc$.
The above sequence is {\it homogeneous} if $V_i$ and $W_i$
 are homogeneous elements of~$\Rc$ with opposite parity.

\begin{defn}
The supersymmetric shift operator is the linear operator
\begin{equation}
\label{OperT}
\gT=\frac{\partial}{\partial{}\xi}-\xi\,T,
\end{equation}
where $T$ is the usual shift operator.
More explicitly, 
the action of $\gT$ on sequences is given by
$$
\gT\left(V+\xi{}W\right)_i=W_i-\xi{}V_{i-1}.
$$
\end{defn}

\begin{rem}
The operator $\gT$ can be viewed as a discrete version (or ``exponential'')
of the odd vector field 
$$
D=\frac{\partial}{\partial{}\xi}-\xi\frac{\partial}{\partial{}x},
$$
satisfying $D^2=-\frac{\partial}{\partial{}x}$.
This vector field is sometimes called the ``SUSY-structure'', or the contact structure
in dimension $1|1$; for more details, see~\cite{Man,Lei,LeiR}.
The vector field $D$ is characterized as the unique odd left-invariant vector field
on the abelian supergroup $\R^{1|1}$, see Appendix.
We believe that the operator $\gT$ is a natural 
discrete analog of $D$ because of the following properties
(that can be checked directly):

\begin{enumerate}
\item[(i)]
One has $\gT^2=-T$.

\item[(ii)]
The operator $\gT$ is equivariant
with respect to the following action of $\Z\oplus\Rc_{\bar1}$
on sequences in $\widehat \Rc$:
$$
(k,\l):\left(V+\xi{}W\right)_i\longmapsto
V_{i+k}-\l{}W_{i+k}+\xi\left(\l{}V_{i+k-1}+W_{i+k}\right),
$$
which is a discrete version of the (left) action of the supergroup $\R^{1|1}$ on itself, see Appendix.
\end{enumerate}
\end{rem}

\noindent
It is natural to say that $\gT$ 
is a difference operator of {\it order}~$\frac{1}{2}$.

\subsection{Supersymmetric discrete Sturm-Liouville operators}

We introduce a new notion of
{\it supersymmetric discrete Sturm-Liouville operator}
(or one-dimensional Schr\"odinger operator),
and the corresponding recurrence equations. 

\begin{defn}
The supersymmetric discrete Sturm-Liouville
operator with potential $U$ is
the following odd linear difference operator of order~$\frac{3}{2}$:
\begin{equation}
\label{SuperSL}
\cL=\gT^3+U\gT^2+\Pi,
\end{equation}
where $U$ is a given odd sequence:
$$
U_i=\b_i+\xi{}a_i,
$$
with $a_i\in\Rc_{\bar0},\;\b_i\in\Rc_{\bar1}$,
and where $\Pi$ is the standard {\it parity inverting} operator:
$$
\Pi\left(V+\xi{}W\right)_i=W_i+\xi{}V_i.
$$
\end{defn}

More explicitly, the operator $\cL$ acts on sequences as follows
$$
\begin{array}{rcl}
\cL\left(V+\xi{}W\right)_i&=&
W_i-W_{i-1}-\b_iV_{i-1}\\[6pt]
&&\displaystyle
+\xi\left(V_i-a_iV_{i-1}+V_{i-2}+\b_iW_{i-1}\right).
\end{array}
$$

The corresponding linear recurrence equation $\cL\left(V+\xi{}W\right)=0$
is the following system:
$$
\left\{
\begin{array}{rcl}
V_i&=&a_iV_{i-1}-V_{i-2}-\b_iW_{i-1},\\[4pt]
W_i&=&W_{i-1}+\b_iV_{i-1},
\end{array}
\right.
$$
for all $i\in\Z$.
It can be written in the matrix form:
\begin{equation}
\label{SuperHE}
\left(
\begin{array}{l}
V_{i-1}\\[4pt]
V_i\\[4pt]
W_i
\end{array}
\right)=
A_i\left(
\begin{array}{l}
V_{i-2}\\[4pt]
V_{i-1}\\[4pt]
W_{i-1}
\end{array}
\right),
\qquad
\hbox{where}
\qquad
A_i=
\left(
\begin{array}{cc|c}
0&1&0\\[4pt]
-1&a_i&-\b_i\\[4pt]
\hline
0&\b_i&1
\end{array}
\right).
\end{equation}
This is a supersymmetric analog of the equation~(\ref{SLEq}).
It is easy to check that the matrix in the right-hand-side
belongs to the supergroup $\OSp(1|2)$;
see Appendix.

\begin{rem}
The continuous limit of the operator~(\ref{SuperSL}) is
the well-known supersymmetric 
Sturm-Liouville Operator:
$$
D^3+U(x,\xi),
$$
considered by many authors, see, e.g.,~\cite{Rad,GT}.
This differential operator is self-adjoint with respect to the Berezin integration.
It is related to the coadjoint representation
of the Neveu-Schwarz and Ramond superanalogs of the Virasoro algebra;
see~\cite{DMV}.

More precisely, $U(x,\xi)=U_1(x)+\xi{}U_0(x)$,
and in the Neveu-Schwarz case the  function is periodic:
$U_0(x+2\pi)=U_0(x)$ and $U_1(x+2\pi)=U_1(x)$,
while in the Ramond case it is (anti)periodic:
$$
U_0(x+2\pi)=U_0(x),
\qquad U_1(x+2\pi)=-U_1(x).
$$
This corresponds to two different versions of the supercircle,
the one related to the trivial bundle over $S^1$,
and the second one related to the M\"obius bundle.
\end{rem}

\subsection{Supersymmetric Hill equations, monodromy and
 supervariety $\E_n$}

We will always assume the following periodicity condition on the
coefficients of the Sturm-Liouville operator:
\begin{equation}
\label{SPerEq}
a_{i+n}=a_i,
\qquad
\b_{i+n}=-\b_i.
\end{equation}
Periodicity of coefficients does not, of course, imply 
periodicity or antiperiodicity of solutions. 
Any such equation has a {\it monodromy operator}, acting on the space of solutions
$$
\left(
\begin{array}{l}
V_{i+n-1}\\[4pt]
V_{i+n}\\[4pt]
W_{i+n}
\end{array}
\right)=
M
\left(
\begin{array}{l}
V_{i-1}\\[4pt]
V_{i}\\[4pt]
W_{i}
\end{array}
\right).
$$
This operator can be represented as a matrix $M_i$
which is a product of $n$ consecutive matrices:
\begin{equation}
\label{MonEq}
M_i=A_{i+n-1}A_{i+n-2}\cdots{}A_{i+1}A_i,
\end{equation}
where $A_i$ is the matrix 
of the system~(\ref{SuperHE}), and therefore $M_i\in\OSp(1|2)$.

\begin{defn}
\label{DefHEq}
A {\it supersymmetric Hill equation} is the equation
(\ref{SuperHE}) such that
all its solutions $V+\xi{}W=(V_i+\xi{}W_i),\;i\in\Z$ satisfy
the following (anti)periodicity condition:
\begin{equation}
\label{PeriodSol}
V_{i+n}=-V_i,
\qquad
W_{i+n}=W_i,
\end{equation}
for all $i\in\Z$.
\end{defn}

Since the space of solutions has dimension $2|1$,
the condition~\eqref{PeriodSol} is equivalent to the fact that the monodromy matrix of such equation is:
\begin{equation}
\label{MonCond}
M_i=\left(
\begin{array}{rr|c}
-1&0&\;\;0\\[4pt]
0&-1&\;\;0\\[4pt]
\hline
0&0&\;\;1
\end{array}
\right).
\end{equation}

Remarkably enough, the condition~\eqref{MonCond} does not depend on the choice
of the initial $i$.

\begin{lem}
\label{MInd}
If the condition~\eqref{MonCond} holds for some $i$, then it holds for all $i\in\Z$.
\end{lem}

\begin{proof}
Let $M_i$ be as in~\eqref{MonCond} for some $i$.
By definition~\eqref{MonEq},
and using (anti)periodicity of the coefficients~\eqref{SPerEq}, we have:
$$
\begin{array}{rcl}
M_{i+1}=A_{i+n}M_iA_i^{-1}&=&
\left(
\begin{array}{cc|c}
0&1&0\\[4pt]
-1&a_i&\b_i\\[4pt]
\hline
0&-\b_i&1
\end{array}
\right)
\left(
\begin{array}{rr|c}
-1&0&\;\;0\\[4pt]
0&-1&\;\;0\\[4pt]
\hline
0&0&\;\;1
\end{array}
\right)
\left(
\begin{array}{cc|c}
a_i&-1&-\b_i\\[4pt]
1&0&0\\[4pt]
\hline
-\b_i&0&1
\end{array}
\right)\\[22pt]
&=&
\left(
\begin{array}{rr|c}
-1&0&\;\;0\\[4pt]
0&-1&\;\;0\\[4pt]
\hline
0&0&\;\;1
\end{array}
\right).
\end{array}
$$
The result follows by induction.
\end{proof}

The condition~\eqref{PeriodSol} is a strong condition on the potential of Hill's equation.
More precisely, one has the following.

\begin{prop}
\label{HillProp}
The space of Hill's equations satisfying the condition~(\ref{PeriodSol})
is an algebraic supervariety of dimension $\left(n-3\right)|\left(n-2\right)$.
\end{prop}

\begin{proof}
The space of all Hill's equations with arbitrary (anti)periodic potential
is just a vector space of dimensional $n|n$.
The matrix $M$ is given by the product~(\ref{MonEq}), and therefore has
polynomial coefficients in $a$'s and $\b$'s.
Thus, the condition~(\ref{PeriodSol}) defines an algebraic variety.
Furthermore, the condition~(\ref{PeriodSol}) has codimension $3|2$,
i.e., the dimension of~$\OSp(1|2)$.
\end{proof}

We will denote by $\E_n$ the supervariety of Hill's equations
satisfying the condition~(\ref{PeriodSol}).

\begin{rem}
The condition~\eqref{SPerEq} is manifestly a discrete version
of the Ramond superalgebra, i.e., it corresponds to the M\"obius supercircle.
We do not know if the Neveu-Schwarz algebra can be discretized
with the help of linear difference operators.
This case would correspond to periodic sequences $a$'s and $b$'s,
but then the monodromy $M$ would have to be $\pm\Id$.
However, the case $M=\Id$ cannot be related to Coxeter's friezes,
and $M=-\Id$ is not an element of $\OSp(1|2)$.
\end{rem}

\subsection{Supervariety $\E_n$ for small values of $n$}\label{MonEx}

Using the condition~(\ref{MonCond}), one can write down explicitly the
algebraic equations determining the supervariety $\E_n$.
We omit straightforward but long computations.

a)
The supervariety $\E_3$ has dimension $0|1$.
Every Hill's equation satisfying (\ref{PeriodSol}) for $n=3$
has coefficients $a_i\equiv1$, and $\b_i=(-1)^i\b$, for all $i\in\Z$,
where $\b$ is an arbitrary odd variable.
This means that the $\b_i$ satisfy the system:
$$
\left(
\begin{array}{ccc}
0&1&1\\[4pt]
-1&0&1\\[4pt]
-1&-1&0
\end{array}
\right)
\left(
\begin{array}{r}
-\b_3\\[4pt]
\b_1\\[4pt]
\b_2
\end{array}
\right)=0.
$$

b) 
The supervariety $\E_4$ has dimension $1|2$;
the coefficients of the Hill's equation satisfy:
$$
\begin{array}{rcl}
a_1a_2-2+\b_1\b_2&=&0,\\[4pt]
a_2a_3-2+\b_2\b_3&=&0,\\[4pt]
a_3a_4-2+\b_3\b_4&=&0,\\[4pt]
a_4a_1-2+\b_4\b_1&=&0,
\end{array}
\qquad
\hbox{and}
\qquad
\left(
\begin{array}{cccc}
0&1&a_1&1\\[4pt]
-1&0&1&a_2\\[4pt]
-a_1&-1&0&1\\[4pt]
-1&-a_2&-1&0
\end{array}
\right)
\left(
\begin{array}{r}
-\b_4\\[4pt]
\b_1\\[4pt]
\b_2\\[4pt]
\b_3
\end{array}
\right)=0.
$$
The matrix of the linear system has rank $2$.

c)
One can check by a direct computation that
the supervariety $\E_5$ is the $2|3$-dimensional supervariety
defined by the following polynomial equations on $5$ even and $5$ odd variables
$(a_1,\ldots,a_5,\b_1,\ldots,\b_5)$:
$$
\begin{array}{rcl}
a_1a_2-a_4-1+\b_1\b_2&=&0,\\[4pt]
a_2a_3-a_5-1+\b_2\b_3&=&0,\\[4pt]
a_3a_4-a_1-1+\b_3\b_4&=&0,\\[4pt]
a_4a_5-a_2-1+\b_4\b_5&=&0,\\[4pt]
a_5a_1-a_3-1-\b_5\b_1&=&0,
\end{array}
\qquad
\hbox{and}
\qquad
\left(
\begin{array}{ccccc}
0&1&a_1&a_4&1\\[4pt]
-1&0&1&a_2&a_5\\[4pt]
-a_1&-1&0&1&a_3\\[4pt]
-a_4&-a_2&-1&0&1\\[4pt]
-1&-a_5&-a_3&-1&0
\end{array}
\right)
\left(
\begin{array}{r}
-\b_5\\[4pt]
\b_1\\[4pt]
\b_2\\[4pt]
\b_3\\[4pt]
\b_4
\end{array}
\right)=0.
$$
Notice that exactly $3$ even and $2$ odd equations are independent.

d)
The supervariety $\E_6$ is determined by six ``even'' equations
on variables $(a_1,\ldots,a_6,\b_1,\ldots,\b_6)$,
namely:
$$
a_1+a_3+a_5-a_3a_4a_5-a_3\b_4\b_5-a_5\b_3\b_4-\b_3\b_5=0,
$$
and its cyclic permutations, together with the following system
of linear equations: 
$$
\left(
\begin{array}{cccccc}
0&1&a_1&a_1a_2-1&a_5&1\\[8pt]
-1&0&1&a_2&a_2a_3-1&a_6\\[8pt]
-a_1&-1&0&1&a_3&a_3a_4-1\\[8pt]
1-a_1a_2&-a_2&-1&0&1&a_4\\[8pt]
-a_5&1-a_2a_3&-a_3&-1&0&1\\[8pt]
-1&-a_6&1-a_3a_4&-a_4&-1&0
\end{array}
\right)
\left(
\begin{array}{r}
-\b_6\\[8pt]
\b_1\\[8pt]
\b_2\\[8pt]
\b_3\\[8pt]
\b_4\\[8pt]
\b_5
\end{array}
\right)=0.
$$

Observe that
the above equations for the even variables $a_i$ are projected
(modulo $\Rc_{\bar1}$) to the equations defining classical
Coxeter frieze patterns.
The odd variables $\b_i$ in each of the above examples
satisfy a systems of linear equations.
The (skew-symmetric) matrices of the linear systems are
nothing other than the matrices of Coxeter's friezes (for more details, see Section~\ref{CoSec}).

\section{Superfriezes}\label{SFS}

In this section, we introduce the notion of {\it superfrieze}.
It is analogous to that of Coxeter frieze,
and the main properties of superfriezes are similar to those of
Coxeter friezes.
The space of all superfriezes is an algebraic supervariety
isomorphic to the supervariety $\E_n$ of supersymmetric Hill's equations.
In this sense, a frieze as just another, equivalent, way to record
Hill's equations.
Superfriezes provide a good parametrization of the space of Hill's equations.

\subsection{Coxeter frieze patterns and Euler's continuants}\label{CoSec}
We start with an overview of the classical Coxeter frieze patterns,
and explain an isomorphism between the spaces of 
Sturm-Liouville operators and that
of frieze patterns.
For more details, we refer to~\cite{Cox,CoCo,ARS,CaCh,MG,SVRS,Ust}.

The notion of frieze pattern (or a frieze, for short) is due to Coxeter \cite{Cox}.
We define a frieze 
 as an infinite array of numbers 
(or functions, polynomials, etc.):
$$
\begin{array}{cccccccccccc}
&\cdots&&0&&0&&0&&0&&\cdots\\[4pt]
\cdots&&1&&1&&1&&1&&\cdots\\[4pt]
&\cdots&&a_{i}&&a_{i+1}&&a_{i+2}&&a_{i+3}&&\cdots\\[4pt]
&& \cdots&& \cdots&& \cdots&& \cdots&&
\end{array}
$$
where the entries of each next row are  
determined by the previous two rows via the following frieze rule:
for each elementary ``diamond''
$$
\begin{array}{ccc}
&b&\\[2pt]
a&&d\\[2pt]
&c&
\end{array}
$$
one has
$
ad-bc=1.
$

For instance, the entries in the next row of the above frieze are
$a_ia_{i+1}-1$, and the following row has the entries 
$a_ia_{i+1}a_{i+2}-a_i-a_{i+2}$. 

Starting from generic values in the first row of the frieze, the frieze rule  defines the next rows.
For a {\it generic} frieze, the entries of the $k$-th row are equal to
the following determinant
$$
K(a_i,\ldots,a_{i+k-1})=
\left|
\begin{array}{cccc}
a_i&1&&\\[4pt]
1&a_{i+1}&1&\\[4pt]
&\ddots&\ddots&1\\[4pt]
&&1&a_{i+k-1}
\end{array}
\right|
$$
which is a classical
{\it continuant}, already considered by Euler, see~\cite{Mui}.

A frieze pattern is called {\it closed} if a row of $1$'s appears again,
 followed by a row of~$0$'s:
$$
\begin{array}{ccccccccccc}
&0&&0&&0&&0&&\cdots\\[4pt]
\cdots&&1&&1&&1&&1&&\cdots\\[4pt]
&a_i&&a_{i+1}&&a_{i+2}&&a_{i+3}&&a_{i+4}\\[4pt]
&& \cdots&& \cdots&& \cdots&& \cdots&&\\[4pt]
\cdots&&1&&1&&1&&1&&\cdots\\[4pt]
&0&&0&&0&&0&&\cdots
\end{array}
$$
The {\it width} $m$ of a closed frieze pattern is the number of non-trivial rows between the rows of~$1$'s.
In other words, a frieze is closed of width $m$, if and only if
$K(a_i,\ldots,a_{i+m})=1$, and $K(a_i,\ldots,a_{i+m+1})=0$, for all $i$.

Friezes introduced and studied by Coxeter \cite{Cox} are exactly the closed friezes.

Let us recall the following results on friezes,~\cite{Cox,CoCo,ARS}:
\begin{enumerate}
\item
A closed frieze pattern is horizontally periodic with period
$n=m+3,$ that is, $a_{i+n}=a_i$. 
\item
Furthermore, a closed frieze pattern has ``glide symmetry'' whose second iteration 
is the horizontal parallel translation of distance~$n$. 
\item
A frieze pattern with the first row $(a_i)$ is closed if and only if
the Sturm-Liouville equation with potential $(a_i)$ 
has antiperiodic solutions.
\end{enumerate}

The name ``frieze pattern" is due to the glide symmetry.

\begin{prop}
\label{CloP}
The space of closed friezes of width $m$ is an {\it algebraic variety}
of dimension~$m$.
\end{prop}

\noindent
Indeed, a closed frieze is periodic, so that one has
a total of $2n=2m+6$ algebraic equations,
$K(a_i,\ldots,a_{i+m})=1$  and $K(a_i,\ldots,a_{i+m+1})=0$
in $n$ variables $a_1,\ldots,a_n$.
It turns out that exactly~$3$ of these equations are algebraically independent,
and imply the rest. 

Based on results of \cite{Cox}, one can formulate the following statement:
the two algebraic varieties below are isomorphic:
\begin{enumerate}
\item
the space of Sturm-Liouville equations with $n$-antiperiodic solutions;
\item
the space of Coxeter's frieze patterns of width $m=n-3$;
 \end{enumerate}
 for details, see \cite{MGOT}, \cite{SVRS}.
The idea of the proof is based on the fact that every diagonal
of a frieze pattern is a solution to the Sturm-Liouville equation
with potential $(a_i),\;i\in\Z$.
More precisely, one has the recurrence formula for continuants
$$
K(a_i,\ldots,a_j)=a_jK(a_i,\ldots,a_{j-1})
-K(a_i,\ldots,a_{j-2})
$$
already known to Euler.

The above isomorphism allows one to identify Hill's equations
and frieze patterns.
The main interest in associating a frieze to a given
Sturm-Liouville equation is that the frieze provides remarkable
local coordinate systems. The coordinates are known as ``cluster coordinates''.

\begin{ex}
\label{CEx}
{\rm
A generic Coxeter frieze pattern of width~$2$ is as follows:
$$
 \begin{array}{ccccccccccc}
\cdots&&1&& 1&&1&&\cdots
 \\[4pt]
&x_1&&\frac{x_2+1}{x_1}&&\frac{x_1+1}{x_2}&&x_2&&
 \\[4pt]
\cdots&&x_2&&\frac{x_1+x_2+1}{x_1x_2}&&x_1&&\cdots
 \\[4pt]
&1&&1&&1&&1&&
\end{array}
$$
for some $x_1,x_2\not=0$.
(Note that we omitted the first and the last rows of $0$'s.)}
This example is related to the work of Gauss~\cite{Gau} on so-called
{\it Pentagramma Mirificum}.
It was noticed by Coxeter~\cite{Cox} that the values of various elements of self-dual spherical pentagons, 
calculated by Gauss, form a frieze of width $2$.
\end{ex}

\subsection{Introducing superfrieze}\label{TheDef}

Similarly to the case of classical Coxeter's friezes, a superfrieze
is a horizontally-infinite array bounded by rows of $0$'s and $1$'s.
Even and odd elements alternate and form ``elementary diamonds'';
there are twice more odd elements.

\begin{defn}
\label{TheMainDefn}
A superfrieze, or a supersymmetric frieze pattern, is the following array
$$
\begin{array}{ccccccccccccccccccccccccc}
&\ldots&0&&&&0&&&&0\\[10pt]
\ldots&{\color{red}0}&&{\color{red}0}&&{\color{red}0}
&&{\color{red}0}&&{\color{red}0}&&\ldots\\[10pt]
\;\;\;1&&&&1&&&&1&&&\ldots\\[10pt]
&{\color{red}\varphi_{0,0}}&&{\color{red}\varphi_{\frac{1}{2},\frac{1}{2}}}&&{\color{red}\varphi_{1,1}}
&&{\color{red}\varphi_{\frac{3}{2},\frac{3}{2}}}&&{\color{red}\varphi_{2,2}}&&\ldots\\[12pt]
&&f_{0,0}&&&&f_{1,1}&&&&f_{2,2}\\[10pt]
&{\color{red}\varphi_{-\frac{1}{2},\frac{1}{2}}}&&{\color{red}\varphi_{0,1}}
&&{\color{red}\varphi_{\frac{1}{2},\frac{3}{2}}}
&&{\color{red}\varphi_{1,2}}&&{\color{red}\varphi_{\frac{3}{2},\frac{5}{2}}}&&\ldots\\[10pt]
f_{-1,0}&&&&f_{0,1}&&&&f_{1,2}&&\\[4pt]
&\iddots&&\iddots&& \ddots&&\ddots&& \ddots&&\!\!\!\ddots\\[4pt]
&&f_{2-m,1}&&&&f_{0,m-1}&&&&f_{1,m}&&&&\\[10pt]
\ldots&{\color{red}\varphi_{\frac{3}{2}-m,\frac{3}{2}}}&&{\color{red}\varphi_{2-m,2}}&&\ldots
&&{\color{red}\varphi_{0,m}}&&{\color{red}\varphi_{\frac{1}{2},m+\frac{1}{2}}}&&{\color{red}\varphi_{1,m+1}}\\[10pt]
\;\;\;1&&&&1&&&&1&&&&&\\[10pt]
\ldots&{\color{red}0}&&{\color{red}0}&&{\color{red}0}
&&{\color{red}0}&&{\color{red}0}&&{\color{red}0}&\\[10pt]
&\ldots&0&&&&0&&&&0&\ldots
\end{array}
$$
where $f_{i,j}$ are even and $\varphi_{i,j}$ are odd, and where every 
{\it elementary diamond}:
$$
\begin{array}{ccccc}
&&B&&\\[4pt]
&{\color{red}\Xi}&&{\color{red}\Psi}&\\[4pt]
A&&&&D\\[4pt]
&{\color{red}\Phi}&&{\color{red}\Sigma}&\\[4pt]
&&C&&
\end{array}
$$
satisfies the following conditions:
\begin{equation}
\label{Rule}
\begin{array}{rcl}
AD-BC&=&1+\Sigma\Xi,\\[4pt]
A\Sigma-C\Xi&=&\Phi,\\[4pt]
B\Sigma-D\Xi&=&\Psi,
\end{array}
\end{equation}
that we call the {\it frieze rule}.

The integer $m$, i.e., the number of even rows between the rows
of $1$'s is called the {\it width} of the superfrieze.
\end{defn}

\begin{rem}
As usual, in the ``supercase'' there exists a projection to the
classical case.
Indeed, choosing all the odd variables $\varphi_{i,j}=0$,
the above definition is equivalent to the definition of a classical Coxeter frieze pattern
with entries $f_{i,j}$.
\end{rem}

Let us comment on the notation.
The indices $i,j$ of the entries of the frieze
stand to number of diagonals of the frieze.
More precisely, the first index $i$ numbers South-East diagonals,
and the second index $j$ numbers North-East diagonals.

\subsection{More about the frieze rule}

The last two equations of~(\ref{Rule}) are equivalent to
$$
B\Phi-A\Psi=\Xi,
\qquad
D\Phi-C\Psi=\Sigma.
$$
Note also that these equations also imply $\Xi\Sigma=\Phi\Psi$,
so that the first equation of~(\ref{Rule}) can also be written as follows: 
$$
AD-BC=1+\Psi\Phi.
$$

Another way to express the last two equations of the frieze rule is to consider
the odd entries neighboring the elementary diamond.
Then for every configuration 
$$
\begin{array}{ccccccc}
&&{\color{red}\widetilde{\Psi}}&&{\color{red}\widetilde{\Xi}}\\[4pt]
&&&B&&\\[4pt]
{\color{red}\widetilde{\Phi}}&&{\color{red}\Xi}&&{\color{red}\Psi}&&{\color{red}\widetilde{\Sigma}}\\[4pt]
&A&&&&D\\[4pt]
&&{\color{red}\Phi}&&{\color{red}\Sigma}&\\[4pt]
&&&C&&
\end{array}
$$
of the frieze, one has:
$$
B\,(\Phi-\widetilde{\Phi})=
A\,(\Psi-\widetilde{\Psi}),
\qquad
B\,(\Sigma-\widetilde{\Sigma})=
D\,(\Xi-\widetilde{\Xi}).
$$

The relation to the group $\OSp(1|2)$ is as follows.
One can associate an elementary diamond with every element of the supergroup $\OSp(1|2)$
(see Appendix)
using the following formula:
$$
\left(
\begin{array}{cc|c}
a&b&\g\\[4pt]
c&d&\d\\[4pt]
\hline
\a&\b&e
\end{array}
\right)
\qquad
\longleftrightarrow
\qquad
\begin{array}{ccccc}
&&\!\!\!-a&&\\[4pt]
&{\color{red}\g}&&{\color{red}\a}&\\[4pt]
b&&&&\!\!\!\!-c\\[4pt]
&\!\!\!{\color{red}-\b}&&{\color{red}\d}&\\[4pt]
&&d&&
\end{array}
$$
so that the relations~(\ref{Rule}) coincide with the relations
defining an element of $\OSp(1|2)$.

The frieze rule~(\ref{Rule}) implies the following elementary but useful properties.

\begin{prop}
\label{EasyOne}
(i)
The entries $\varphi_{i,i}$ in the first non-trivial row
of a generic superfrieze consist of pairs of equal ones:
$\varphi_{i,i}=\varphi_{i+\frac{1}{2},i+\frac{1}{2}}$, where $i\in\Z$.

(ii) 
The entries $\varphi_{i,i}$ in the last non-trivial row of a generic superfrieze consist of pairs of opposite ones:
$\varphi_{i,i+m}=-\varphi_{i-\frac{1}{2},i+m-\frac{1}{2}}$, where $i\in\Z$.
\end{prop}

\subsection{Examples of superfriezes}\label{SecESF}

The generic superfrieze of width $m=1$ is of the following form:
$$
\begin{array}{ccccccccccccccccccccccc}
&&0&&&&0&&&&\!\!0&&&&0\\[8pt]
&{\color{red}0}&&{\color{red}0}&&{\color{red}0}&&
\!\!\!{\color{red}0}&&{\color{red}0}&&{\color{red}0}&&\;\;\;{\color{red}0}&&\!\!\!{\color{red}0}\\[8pt]
1&&&&1&&&&1&&&&\!\!1&&&&1\\[8pt]
&{\color{red}\xi}&&{\color{red}\xi}&&{\color{red}\xi'}
&&{\color{red}\xi'}&&{\color{red}\xi-x\eta}&&{\color{red}\xi-x\eta}
&&{\color{red}\eta}&&{\color{red}\eta}\\[10pt]
&&x&&&&x'&&&&\!\!\!x&&&&x'\\[10pt]
&{\color{red}\xi-x\eta}&&\;\;\;{\color{red}x\eta-\xi}&&\;{\color{red}\eta}&&{\color{red}-\eta}
&&{\color{red}-\xi}&&{\color{red}\xi}&&{\color{red}-\xi'}&&{\color{red}\xi'}\\[8pt]
1&&&&1&&&&1&&&&\!\!1&&&&1\\[8pt]
&{\color{red}0}&&{\color{red}0}&&\;\;{\color{red}0}
&&\!{\color{red}0}&&{\color{red}0}&&{\color{red}0}&&\;\;\;{\color{red}0}&&\!\!\!{\color{red}0}\\[10pt]
&&0&&&&0&&&&\!\!0&&&&0\
\end{array}
$$
where
$$
x'=\frac{2}{x}+\frac{\eta\xi}{x},
\qquad 
\xi'=\eta-\frac{2\xi}{x}.
$$
One can choose local coordinates $(x,\xi,\eta)$
to parametrize the space of friezes.

The following example is the superanalog of
the frieze from Example~\ref{CEx} related to
the Gauss Pentagramma mirificum.
$$
\begin{array}{cccccccccccccccccccccc}
0&&&&\!\!\!0&&&&\!\!0&&&&\!\!\!0&&&&0&&\ldots\\[10pt]
&{\color{red}0}&&\!\!{\color{red}0}&&{\color{red}0}
&&{\color{red}0}&&{\color{red}0}&&{\color{red}0}&&\!\!{\color{red}0}&&{\color{red}0}&&{\color{red}0}\\[10pt]
\ldots&&1&&&&1&&&&1&&&&\!\!1&&&&\!\!1\\[4pt]
&{\color{red}-\zeta}
&&\!\!{\color{red}\xi}&&{\color{red}\raisebox{.5pt}{\textcircled{\raisebox{-.3pt} {$\xi$}}}}&&{\color{red}\xi'}&&{\color{red}\xi'}
&&\color{red}{\nu}
&&\color{red}{\nu}
&&\!\!{\color{red}\zeta^*}&&{\color{red}\zeta^*}
\\[10pt]
y'&&&&
\!\!{\raisebox{.5pt}{\textcircled{\raisebox{-.3pt} {$x$}}}}&&&&\!\!x'
&&&&x''&&&&y\\[10pt]
&\!\!\!{\color{red}-\eta'}&&
\!\!\!{\color{red}\eta^*}
&&\color{red}{\tau}
&&{\color{red}\raisebox{.5pt}{\textcircled{\raisebox{-.1pt} {$\eta$}}}}
&&\color{red}{\tau'}
&&{\color{red}\eta'}
&&\!\!\!{\color{red}\eta^*}&&\!\!{\color{red}-\tau}
&&{\color{red}\eta}\\[10pt]
&&x''&&&&{\raisebox{.5pt}{\textcircled{\raisebox{-.1pt} {$y$}}}}
&&&&y'&&&&x
&&&&\!\!\!x'\\[10pt]
&\color{red}{\nu}
&&\color{red}{-\nu}
&&{\color{red}\zeta^*}
&&{\color{red}-\zeta^*}
&&{\color{red}\raisebox{.5pt}{\textcircled{\raisebox{-.3pt} {$\zeta$}}}}
&&\!\!{\color{red}-\zeta}&&\!\!\!{\color{red}-\xi}
&&{\color{red}\xi}&&{\color{red}-\xi'}\\[10pt]
1&&&&\!\!1&&&&1&&&&1&&&&1&&\ldots\\[10pt]
&{\color{red}0}&&\!\!{\color{red}0}&&{\color{red}0}&&{\color{red}0}&&{\color{red}0}
&&{\color{red}0}&&\!\!{\color{red}0}&&{\color{red}0}&&{\color{red}0}\\[10pt]
\ldots&&0&&&&0&&&&0&&&&0&&&&\!\!0
\end{array}
$$
The frieze is defined by the initial values
$(x,y,\xi,\eta,\zeta)$, the next values are easily calculated using the frieze rule.
The even entries of the superfrieze are as follows:
$$
x'=\frac{1+y}{x}+\frac{\eta\xi}{x},
\qquad
y'=\frac{1+x+y}{xy}+\frac{\eta\xi}{xy}+\frac{\zeta\eta}{y},
\qquad
x''=\frac{1+x}{y}+\frac{\eta\xi}{y}+\xi\zeta+\frac{x}{y}\zeta\eta.
$$
For the odd entries of the superfrieze, one has:
$$
\begin{array}{lll}
\displaystyle
\xi'=\eta-\frac{1+y}{x}\xi,
&
\displaystyle
\eta'=\zeta-\frac{1+x+y}{xy}\xi-\frac{\xi\eta\zeta}{y},
&
\displaystyle
\tau'=
\frac{1+y}{x}\zeta-\frac{1+x+y}{xy}\eta-\frac{\xi\eta\zeta}{x},
\end{array}
$$
$$
\begin{array}{llll}
\displaystyle
\zeta^*=\eta-y\zeta,
&
\quad
\displaystyle
\eta^*=\xi-x\zeta,
&
\quad
\displaystyle
\nu=
\frac{(1+x)}{y}\eta-\xi-\zeta,
&
\quad
\displaystyle
\tau=
x\eta-y\xi,
\end{array}
$$

The superfriezes exhibited in the above example have many
symmetries and periodicities.
Our next task is to derive these properties
of superfriezes in general.

\subsection{Generic superfriezes and Hill's equations}\label{RecuRSSec}
Like Coxeter's friezes, superfriezes enjoy remarkable properties,
under some conditions of genericity.
We begin with the most elementary way to define generic superfriezes.

\begin{defn}
\label{GenDef}
A superfrieze is called {\it generic} if every even entry is invertible.
\end{defn}

The following lemma explains the relation between superfriezes and linear difference equations. 

\begin{lem}
\label{RecurLem}
The entries of every South-East diagonal of a generic superfrieze 
$$
(W_i,V_i):=(\varphi_{j,i},\,f_{j,i}),
$$ 
where $j$ is an arbitrary (fixed) integer,
satisfy Hill's equation~(\ref{SuperHE}) with
the potential $U_i=\b_i+\xi{}a_i$, where $a_i$ and $\b_i$ are given by the first two rows
of the superfrieze,
i.e., $a_i=f_{i,i}$ and $\b_i=\varphi_{i,i}$.
\end{lem}

\begin{proof}
We proceed by induction.
Assume that the following fragment of a superfrieze:
$$
\begin{array}{ccccccccc}
&&&&0\\
&&&\iddots&&{\color{red}0}\\[2pt]
&&\;B&&&&1\\[2pt]
&\;{\color{red}\Xi}&&{\color{red}\Psi}&&\iddots&&{\color{red}\b_i}\\[2pt]
A&&&&D&&&&a_i\\[2pt]
&\;{\color{red}\Phi}&&{\color{red}\Sigma}&&{\color{red}\L}&&\iddots\\[2pt]
&&\;C&&&&F\\[2pt]
&&&{\color{red}\Omega}&&\iddots\\[2pt]
&&&&E
\end{array}
$$
satisfies the relations
$$
F=a_iD-B-\b_i\Psi,
\qquad
\L=\Psi+\b_iD
$$
corresponding to the recurrence~\eqref{SuperHE}.
We need to prove that these relations propagate on the next diagonal, i.e., that
$$
E=a_iC-A-\b_i\Phi,
\qquad
\Omega=\Phi+\b_iC.
$$
Indeed, using the superfrieze rule $D(\Omega-\Phi)=C(\L-\Psi)$, we deduce that
$D(\Omega-\Phi)=\b_iCD$, and canceling  $D$, we obtain the second desired relation.
For the even entries, we use the rule:
$CF-DE=1+\L\Omega$ together with $AD-BC=1-\Phi\Psi$.
We have:
$$
\begin{array}{rcl}
DE &=& CF-1-\L\Omega\\[4pt]
&=& a_iCD-CB-\b_iC\Psi-1-\L\Omega\\[4pt]
&=& a_iCD+1-\Phi\Psi-AD-\b_iC\Psi-1-\L\Omega\\[4pt]
&=& a_iCD-(\Phi+\b_iC)\Psi-AD-(\Psi+\b_iD)(\Phi+\b_iC)\\[4pt]
&=& a_iCD-AD-\b_iD\Phi,
\end{array}
$$
and again canceling $D$ we obtain the desired relation.

Note that canceling $D$ twice is possible due the genericity assumption.
\end{proof}

A similar property holds for North-East diagonals.

\begin{lem}
\label{RecurLemDual}
The entries of every North-East diagonal of a generic superfrieze 
$$
(W^*_i,V^*_i):=(\varphi_{i+\frac32,j+\frac12},\,f_{i+2,j}),
$$ 
where $j$ is an arbitrary (fixed) integer,
satisfy the following Hill equation
\begin{equation}
\label{SuperHEDual}
\left(
\begin{array}{l}
V^*_{i-1}\\[4pt]
V^*_{i}\\[4pt]
W^*_{i}
\end{array}
\right)=
\left(
\begin{array}{rr|r}
0&1&0\\[4pt]
-1&a_i&\b_i\\[4pt]
\hline
0&-\b_i&1
\end{array}
\right)
\left(
\begin{array}{l}
V^*_{i-2}\\[4pt]
V^*_{i-1}\\[4pt]
W^*_{i-1}
\end{array}
\right),
\end{equation}
where $a_i=f_{i,i}$ and $\b_i=\varphi_{i+\frac12,i+\frac12}$.
\end{lem}

\begin{proof}
Consider the $j$th North-East diagonal 
$(V'_i,W'_i):=(f_{i,j},\,\varphi_{i+\frac12,j+\frac12})$.
As in the proof of Lemma~\ref{RecurLem}, by induction one establishes
the following system:
$$
\begin{array}{rcl}
V'_i&=&a_iV'_{i+1}-V'_{i+2}+\b_iW'_{i+1},\\[4pt]
W'_i&=&\b_iV'_{i+1}+W'_{i+1}.
\end{array}
$$
Inverting the matrix of the system and shifting the indices,
one obtains~\eqref{SuperHEDual}.
\end{proof}

Note that the difference between the equation~\eqref{SuperHE} and the equation~\eqref{SuperHEDual}
is in the sign of the odd coefficients $\b_i$.

The following properties are crucial for the notion of variety of friezes
introduced in the sequel.

\begin{prop}
\label{Easy}
(i)
A generic superfrieze is completely determined by the first two non-trivial rows, 
$\varphi_{i,i}$ and $f_{i,i}$, below the row of $1$'s.

(ii)
The entries $f_{j,i}$,  $\varphi_{j,i}$ and $\varphi_{j-\frac12,i+\frac12}$ of a generic superfrieze are polynomials in the
entries $\b_i$ and $a_i$ of the first two rows, defined by the recurrent formula:
\begin{equation}
\label{RecForm}
\left(
\begin{array}{l}
f_{j,i-1}\\[2pt]
f_{j,i}\\[2pt]
\varphi_{j,i}
\end{array}
\right)=
A_i
\left(
\begin{array}{c}
f_{j,i-2}\\[2pt]
f_{j,i-1}\\[2pt]
\varphi_{j,i-1}
\end{array}
\right),
\end{equation}
where $A_i$ is the matrix 
of the system~(\ref{SuperHE}),
starting from the initial conditions
\begin{equation}
\label{RecInitForm}
(f_{j,j-3},f_{j,j-2},\varphi_{j,j-2})=(-1,0,0),
\end{equation}
and $\varphi_{j-\frac12,i+\frac12}=f_{j,i}\varphi_{j-1,i}-f_{j-1,i}\varphi_{j,i}$. 
\end{prop}
\begin{proof}
Lemma~\ref{RecurLem} implies that every diagonal of a generic superfrieze
is determined by~$\b_i$ and~$a_i$ via the Hill equation~(\ref{SuperHE}).
Therefore, the entries of the frieze are
obtained as solutions $(V_i+\xi{}W_i)$ with the initial conditions
$$
V_{-1}=0,
\quad
V_0=1,
\qquad
W_0=0.
$$
Finally, these initial conditions imply $V_{-2}=-1,W_{-1}=0$.
Hence the result. 
\end{proof}

\begin{ex}
\label{PolyEx}
A generic superfrieze starts as follows:
$$
\begin{array}{cccccccccccccc}
&&\ldots&&0&&\ldots\\[10pt]
&&&{\color{red}0}&&{\color{red}0}\\[10pt]
&&1&&&&1\\[10pt]
&{\color{red}\b_0}&&{\color{red}\b_1}&&{\color{red}\b_1}&&{\color{red}\b_2}\\[10pt]
a_0&&&&a_1&&&&a_2\\[10pt]
&
{\color{red}
\begin{array}{c}
a_0\b_1\\
+\b_0
\end{array}}
&&{\color{red}
\begin{array}{c}
a_1\b_0\\
+\b_1
\end{array}}
&&{\color{red}
\begin{array}{c}
a_1\b_2\\
+\b_1
\end{array}}
&&{\color{red}
\begin{array}{c}
a_2\b_1\\
+\b_2
\end{array}}\\[10pt]
&&
\begin{array}{c}
a_0a_1-1\\
+\b_0\b_1
\end{array}
&&&&
\begin{array}{c}
a_1a_2-1\\
+\b_1\b_2
\end{array}\\[10pt]
&&&
{\color{red}
\begin{array}{r}
\b_0\b_1\b_2\\
+\b_0-\b_2\\
+a_0a_1\b_2\\
+a_0\b_1
\end{array}}
&&
{\color{red}
\begin{array}{r}
\b_0\b_1\b_2\\
-\b_0+\b_2\\
+a_1a_2\b_0\\
+a_2\b_1
\end{array}}\\[20pt]
&&\vdots&&
\begin{array}{r}
a_0a_1a_2-a_0\\
+\b_0\b_2-a_2\\
+a_0\b_1\b_2\\
+a_2\b_0\b_1\\
\end{array}
&&\vdots\\
&&&&
\vdots&&
\end{array}
$$
that can be deduced directly from~\eqref{Rule}.
\end{ex}

The fact that the frieze is closed, i.e., ends with the rows of $1$'s and $0$'s, imposes
strong conditions on the values of $(\b_i)$ and $(a_j),\;i,j\in\Z$.
These conditions will be described in Section~\ref{GladSec}.

\subsection{The glide symmetry and periodicity}\label{GladSec}
The properties of periodicity and glide symmetry
are analogous to Coxeter's glide symmetry of frieze patterns.
In the classical case, it was proved by Coxeter~\cite{Cox}.
This periodicity is usually considered in contemporary works
as an illustration of Zamolodchikov's periodicity conjecture;
see~\cite{Kel} and references therein. 

\begin{lem}
\label{PerCol}
The entries of a superfrieze
of width $m$ satisfy the following periodicity property:
$$
\varphi_{i+n,j+n}=-\varphi_{i,j},
\qquad
f_{i+n,j+n}=f_{i,j},
\qquad
\hbox{for all}
\quad
i,j\in\Z,
$$ 
where $n=m+3$;
in particular, the entries of the first two rows satisfy 
$a_{i+n}=a_i,\,\b_{i+n}=\b_i$, for all $i\in\Z$.
\end{lem}

\begin{proof}
Let us first prove that the entries of
the first two non-trivial rows $a_i=f_{i,i}$ and $\b_i=\varphi_{i,i}$ are $n$-(anti)periodic.
Indeed, consider the bottom part of the frieze:
$$
\begin{array}{ccccccccccccccccc}
&&&&\!\!\!\!\!\!\!\!f_{i-m-1,i-2}&&&&&&&\\[10pt]
&&&\!\!\!\!\!\!\!\!{\color{red}\varphi_{i-m-\frac32,i-\frac32}}
&&\!\!\!\!\!\!\!\!{\color{red}\varphi_{i-m-1,i-1}}&&&\\[10pt]
&&\!\!\!\!\!\!\!\!f_{i-m-2,i-2}&&&&\!\!\!\!\!\!\!\!f_{i-m-1,i-1}&&&&\\[10pt]
&\!\!\!\!\!\!\!\!{\color{red}\varphi_{i-m-\frac52,i-\frac32}}&&\!\!{\color{red}0}
&&\!\!{\color{red}0}&&\!\!\!\!\!\!\!\!{\color{red}\varphi_{i-m-1,i}}&&\\[10pt]
f_{i-m-3,i-2}&&&&\!\!0&&&&\!\!\!\!\!\!\!\!f_{i-m-1,i}
\end{array}
$$
We use the odd ``South-East relation''
of Lemma~\ref{RecurLem} with $j=i-m-1$:
$$
\underbrace{\varphi_{i-m-1,i}}_{=0}=
\varphi_{i-m-1,i-1}+\b_i\underbrace{f_{i-m-1,i-1}}_{=1}
$$
to obtain $\varphi_{i-m-1,i-1}=-\b_i$.
We use the odd ``North-East relation''
of Lemma~\ref{RecurLemDual} with $j=i-2$:
$$
\varphi_{i-m-\frac32,i-\frac32}=
-\b_{i-m-3}\underbrace{f_{i-m-2,i-2}}_{=1}+
\underbrace{\varphi_{i-m-\frac52,i-\frac32}}_{=0}
$$
to obtain $\varphi_{i-m-\frac32,i-\frac32}=-\b_{i-m-3}$.
By Proposition~\ref{EasyOne}, Part (ii),
one deduces the antiperiodicity of the odd coefficients $\b_i=-\b_{i-m-3}$.
Similarly, using the even relations, one deduces the periodicity of 
the even coefficients $a_i=a_{i-m-3}$.

Since the first two non-trivial rows determine the frieze,
see Proposition~\ref{Easy}, Part (i),
the periodicity follows.
\end{proof}

Furthermore, the following statement is analogous to the glide symmetry of friezes
discovered by Coxeter~\cite{Cox}.

\begin{thm}
\label{Glade}
A generic superfrieze satisfies the following glide symmetry
\begin{equation}
\label{GladeEq}
\begin{array}{rcl}
f_{i,j}&=&f_{j-m-1,i-2},\\[4pt]
\varphi_{i,j}&=&\varphi_{j-m-\frac{3}{2},i-\frac32},\\[4pt]
\varphi_{i+\frac12,j+\frac12}&=&-\varphi_{j-m-1,i-1},
\end{array}
\end{equation}
for all $i,j\in\Z$.
\end{thm}

\begin{proof}
This statement readily follows from Lemmas~\ref{RecurLem}, \ref{RecurLemDual} and~\ref{PerCol}.

Indeed, choosing $j=1$, the South-East diagonal $(W_i,V_i)=(\varphi_{1,i},\,f_{1,i})$
is determined by the recurrence~\eqref{SuperHE} and the initial condition
$$
V_{-1}=0,
\quad
V_0=1,
\qquad
W_0=0.
$$
On the other hand, choosing $j=m+2$, the North-East diagonal 
$(W^*_i,V^*_i)=(\varphi_{i+\frac32,m+\frac52},\,f_{i+2,m+2}),$
is determined by the recurrence~\eqref{SuperHEDual} and the same initial condition
$$
V^*_{-1}=0,
\quad
V^*_0=1,
\qquad
W^*_0=0.
$$
The two recurrence relations differ only by the sign of the odd coefficients,
therefore one has $(W_i,V_i)=(-W^*_i,V^*_i)$.
The arguments for arbitrary $j$ are similar, and we obtain
$$
(\varphi_{j,i},\,f_{j,i})=(-\varphi_{i+\frac32,m+j+\frac32},\,f_{i+2,m+j+1}).
$$
Finally, using the antiperiodicty of the whole pattern established in Lemma~\ref{PerCol}, 
we deduce the set of relations \eqref{GladeEq}.
\end{proof}

The above statement can be illustrated by the following diagram representing 
the diagonals of the superfrieze:
$$
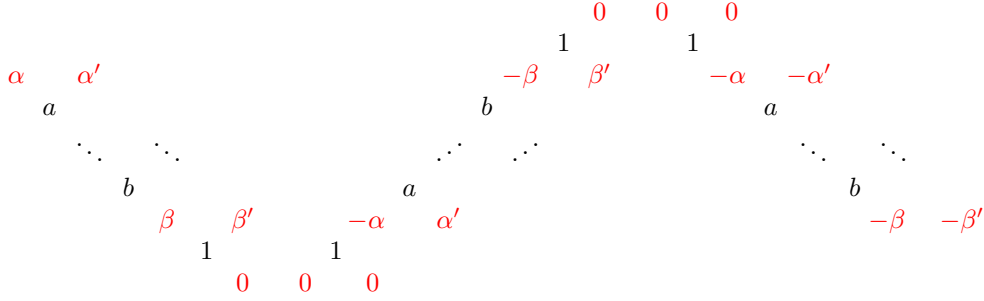
\begin{array}{ccccccccccccccccccccccccccc}
&&&&&&&&&&&&&&&&\!\!{\color{red}0}&&\!\!{\color{red}0}&&\!\!{\color{red}0}\\
&&&&&&&&&&&&&&&\!\!1&&&&\!\!1\\
\!\!{\color{red}\a}&&\!\!{\color{red}\a'}&&&&&&&&&&&&
\!\!\!\!{\color{red}-\b}&&\!\!{\color{red}\b'}&&&&\!\!\!\!{\color{red}-\a}&&\!\!\!\!{\color{red}-\a'}\\
&\!\!a&&&&&&&&&&&&\!\!b&&&&&&&&\!\!a&\\
&&\!\!\ddots&&\!\!\ddots&&&&&&&&\!\!\iddots&&\!\!\iddots&&&&&&&&\!\!\ddots&&\!\!\ddots\\
&&&\!\!b&&&&&&&&\!\!a&&&&&&&&&&&&\!\!b\\
&&&&\!\!{\color{red}\b}&&\!\!{\color{red}\b'}&&&&\!\!\!\!\!{\color{red}-\a}&&\!\!{\color{red}\a'}
&&&&&&&&&&&&\!\!\!\!\!{\color{red}-\b}&&\!\!\!\!\!{\color{red}-\b'}\\
&&&&&\!\!1&&&&\!\!1&&&&\\
&&&&&&\!\!{\color{red}0}&&\!\!{\color{red}0}&&\!\!{\color{red}0}&&&&&&\\
\end{array}
$$

\subsection{The algebraic variety of superfriezes: isomorphism with~$\E_m$}\label{ISoSec}

The above properties of generic superfriezes motivate the following important definition
of the space of superfriezes that includes generic ones.

\begin{defn}
\label{AlgVSF}
The algebraic supervariety of superfriezes is the supervariety defined by
$2n$ even and $n$ odd polynomial equations in variables $(a_1,\ldots,a_n,\b_1,\ldots\b_n)$
expressing that the last three rows of the superfrieze consist in $1$'s and $0$'s.
More precisely, for all $j\in\Z$ and $m=n-3$, one has:
\begin{equation}
\label{Zamk}
f_{j,j+m}=1,
\qquad
f_{j,j+m+1}=0,
\qquad
\varphi_{j,j+m+1}=0,
\end{equation}
where $f_{j,i}$ and $\varphi_{j,i}$
are the polynomials defined by Proposition~\ref{Easy}, Part (ii).
\end{defn}

\noindent
Note that equations~\eqref{Zamk}, together with the frieze rule,
immediately imply $\varphi_{j+\frac12,j+m+\frac32}=0$.

It turns out that the algebraic supervarieties of friezes and that of supersymmetric 
Hill's equations~(\ref{SuperHE}) can be identified.

\begin{thm}
\label{ISOMTH}
The space of superfriezes of width $m$ is an algebraic supervariety
isomorphic to the supervariety $\E_{m+3}$.
\end{thm}

\begin{proof}
By definition of $f_{j,i}$ and $\phi_{j,i}$, one has for all $i,j\in \Z$,
\begin{equation}\label{use}
\left(
\begin{array}{l}
f_{j,i+n-2}\\[2pt]
f_{j,i+n-1}\\[2pt]
\varphi_{j,i+n-1}
\end{array}
\right)=
M_i
\left(
\begin{array}{c}
f_{j,i-2}\\[2pt]
f_{j,i-1}\\[2pt]
\varphi_{j,i-1}
\end{array}
\right),
\end{equation}
 where $M_i$ is as in~\eqref{MonEq}.

Given Hill's equation with potential $\b_i+\xi{}a_i$,
the condition that the monodromy is as in~\eqref{MonCond} 
implies the relations~\eqref{Zamk}, by substituting $i=j-1$ into~\eqref{use}, and using~\eqref{RecInitForm}.

Conversely, assume that the variables $(a_1,\ldots,a_n,\b_1,\ldots\b_n)$ satisfy
the relations~\eqref{Zamk}.
Substituting $i=j-1$ and next $i=j$ into~\eqref{use}, one obtains, respectively
$$
M_{j}
\left(
\begin{array}{r}
\!\!-1\\
0\\
0
\end{array}
\right)
=\left(
\begin{array}{c}
1\\
0\\
0
\end{array}
\right),
\qquad
M_{j}
\left(
\begin{array}{r}
0\\
\!\!-1\\
0
\end{array}
\right)
=\left(
\begin{array}{r}
0\\
1\\
0
\end{array}
\right).
$$
for all $j\in\Z$.
Hence, $M_j$ is of the form
$$
M_j=
\left(
\begin{array}{rr|c}
-1&0&\g\\[4pt]
0&-1&\d\\[4pt]
\hline
0&0&e
\end{array}
\right).
$$
Finally, since $M_j\in \OSp(1|2)$, one deduces that $M_j$ is as in~\eqref{MonCond}. \end{proof}

Proposition~\ref{HillProp} now implies the following.

\begin{cor}
\label{DimProp}
The space of superfriezes of width $m$ is an algebraic supervariety
of dimension $m|m+1$.
\end{cor}

\subsection{Explicit bijection}
Given Hill's equation~(\ref{SuperHE}), let us define the corresponding superfrieze.
Fix $j\in\Z$, and choose a solution $(V^j_i+\xi{}W^j_i)$, $i\in\Z$, with the initial conditions
$
(V^j_{j-2}, W^j_{j-1}, V^j_{j-1})=
(0,0,1).
$
Form the superfrieze defined by 
$$
(\varphi_{j,i},f_{j,i}):=(W^j_i,V^j_i),
$$ 
for all $i,j\in\Z$.
Note that the odd entries with half-integer indices are defined by the frieze rule:
$$
\varphi_{j-\frac12,i+\frac12}=f_{j,i}\varphi_{j-1,i}-f_{j-1,i}\varphi_{j,i}.
$$
The chosen initial condition implies that $V^j_{j-3}=-1$ and $W_{j-2}=0$.

The (anti)periodicity condition~\eqref{PeriodSol} then reads:
$$
V^j_{j+n-3}=f_{j,j+n-3}=1,
\qquad
W^j_{j+n-2}=\varphi_{j,j+n-2}=0,
\qquad
V^j_{j+n-2}=f_{j,j+n-2}=0.
$$
Therefore, we obtain a point of the supervariety of superfriezes.

\subsection{Laurent phenomenon for superfriezes}
The following Laurent phenomenon occurs in the Coxeter friezes:
every entry can be expressed as a Laurent polynomial in the entries of any diagonal.
Example \ref{CEx} illustrates this property.
Similar phenomenon occurs in the
 superfriezes.

\begin{prop}
\label{Laurent}
Every entry of any superfrieze is a Laurent polynomial in 
the entries of any diagonal.
\end{prop}

\begin{proof}
Let us fix a South-East diagonal $(W_i,V_i)$ in the superfrieze.
By Lemma~\ref{RecurLem}, one can express the first rows as
$$
\b_i=\frac{W_{i-1}-W_i}{V_{i-1}},
\qquad
a_i=\frac{V_i+V_{i-2}-\b_iW_{i-1}}{V_{i-1}}.
$$
Therefore, the first two rows are Laurent polynomials in $(W_i,V_i),\;i\in\Z$.
All the entries of the superfrieze are polynomials in the first two rows.
Hence the proposition.
\end{proof}

\section{Open problems}\label{OPSEc}

Here we formulate a series of problems naturally arising in the 
study of superfriezes and supersymmetric difference equations.

\subsection{Supervariety $\E_n$}

Our first two problems concern an explicit form of the equations
characterizing the supervariety $\E_n$.

\begin{pb}
\label{PbOne}
Determine the formula for the entries of a superfrieze\footnote{A solution
to this problem has been given by Alexey Ustinov, see Appendix~2.}.
\end{pb}

In other words, the problem consists in calculating ``supercontinuants''.
The first examples are:
$$
\begin{array}{l}
K(a_i,\b_i)=a_i,
\qquad
K(a_i,a_{i+1},\b_i,\b_{i+1})=a_ia_{i+1}-1+\b_i\b_{i+1},\\[6pt]
K(a_i,\ldots,\b_{i+2})=
a_ia_{i+1}a_{i+2}-a_{i}-a_{i+2}+
\b_ia_{i+1}a_{i+2}+a_ia_{i+1}\b_{i+2}+\b_i\b_{i+2},
\end{array}
$$
cf. Example~\ref{PolyEx}.
Is there a determinantal formula (using Berezinians)
analogous to the classical continuants?

The next question concerns the odd entries of a superfrieze.

\begin{pb}
Do the odd variables $\b_i$ of the first odd row 
satisfy a system of linear equations generalizing the systems of Section~\ref{MonEx}?
\end{pb}

Examples considered in Section~\ref{MonEx} show that, for small values of $n$,
the variables $\b_i$ satisfy linear systems with matrices given by the purely even
Coxeter frieze patterns obtained by projection of superfriezes.

In this paper, we do not investigate the geometric meaning of
superfriezes and Hill's equations.
Recall that classical Hill's equations and Coxeter's friezes are related to the spaces 
$\mathcal{M}_{0,n}$;
see Section~\ref{ClSec} and~\cite{SVRS,Sop}.
We believe that the situation is similar in the supercase.

\begin{pb}
Does the algebraic supervariety $\E_n$ contain the superspace $\gM_{0,n}$
(see~\cite{Wit}) as an open dense subvariety?
\end{pb}

\noindent
An important role in the geometric interpretation of
superfriezes must be played by the super cross-ratio, see e.g.~\cite{MD}.

\subsection{Operators of higher orders}

In this paper, we do not consider the general theory of supersymmetric
difference operators.
We believe that such a theory can be constructed with the help
of the shift operator $\gT$, see formula~(\ref{OperT}),
and formulate here a problem to develop such a theory in full generality.
The corresponding theory of superfriezes must generalize the notion
of $\SL_k$-friezes, see~\cite{ARS,SVRS}.

To give an example, we investigate the next interesting case after the 
Sturm-Liouville operators, namely the operators of order $\frac{5}{2}$.
We omit the details of computations.

In the continuous case, the operators we consider are of the form
$$
D^5+F(x,\xi)D+G(x,\xi),
$$
where the $\Rc$-valued functions $F(x,\xi)$ and $G(x,\xi)$
(for some supercommutative ring $\Rc=\Rc_{\bar0}\oplus\Rc_{\bar1}$) are 
even and odd, respectively.

The discrete version is as follows
$$
\gT^5+\gT^4+U\gT^3+V\gT^2-\Pi,
$$
and the corresponding equation written in the matrix form is as follows:
$$
\left(
\begin{array}{l}
V_{i-2}\\[4pt]
V_{i-1}\\[4pt]
V_i\\[4pt]
W_{i-1}\\[4pt]
W_i
\end{array}
\right)=
\left(
\begin{array}{ccc|cc}
0&1&0&0&0\\[4pt]
0&0&1&0&0\\[4pt]
1&-a_i'&a_i&0&\b_i\\[4pt]
\hline
0&0&0&0&1\\[4pt]
0&0&\b_i'&-1&a_i'-1
\end{array}
\right)
\left(
\begin{array}{l}
V_{i-3}\\[4pt]
V_{i-2}\\[4pt]
V_{i-1}\\[4pt]
W_{i-2}\\[4pt]
W_{i-1}
\end{array}
\right),
$$
where 
$$
F_i=a_i'-1+\xi(\b_i+\b_i'),
\qquad
G_i=\b_i'+\xi{}a_i,
$$
with $a_i,a'_i$ and $\b_i,\b'_i$ arbitrary periodic coefficients.
The periodicity condition in this case should be:
$$
V_{i+n}=V_i,\qquad
W_{i+n}=-W_i.
$$ 

We conjecture that the above difference equations
correspond to a variant of superfriezes analogous to the $2$-friezes, see~\cite{MGOT,MG}.

\section*{Appendix~1: Elements of superalgebra}

To make the paper self-contained, we briefly describe several elementary notions
of superalgebra and supergeometry used above.
For more details, we refer to the classical sources~\cite{Ber,Lei,Man,Man1}.

\medskip
\noindent
{\bf Supercommutative algebras.}
Let Latin letters denote even variables,
and Greek letters the odd ones.
Consider algebras of polynomials
$\mathbb{K}[x_1,\ldots,x_n,\xi_1,\ldots,\xi_k]$,
where $\mathbb{K}=\R,\C$, or some other supercommutative ring, and where $x_i$ are standard
commuting variables, while the odd variables 
$\xi_i$ commute with $x_i$ and anticommute with each other:
$$
\xi_i\xi_j=-\xi_j\xi_i,
$$ 
for all $i,j$; in particular, $\xi_i^2=0$.
Every supercommutative algebra is a quotient of a polynomial algebra by some ideal.
Every supercommutative algebra is the algebra of regular functions on an
algebraic supervariety (which can be taken for a definition of the latter notion).
Every Lie superalgebra is the algebra of derivations of a supercommutative algebra,
for instance, vector fields are derivations of the algebra
of regular functions on an algebraic supervariety.

An example of supercommutative algebra is the Grassmann algebra
of differential forms on a vector space.
Let $(x_1,\ldots,x_n)$ be coordinates, and $(dx_1,\ldots,dx_n)$
their differentials, one
replaces all the differentials $dx_i$ by the odd variables $\xi_i$,
to obtain an isomorphic algebra.

We often need to calculate rational functions
with odd variables.
The main ingredient is the obvious formula
$(1+\xi)^{-1}=1-\xi$. 
For instance, we have:
$$
\frac{y}{x+\xi}=\frac{y}{x(1+\xi/x)}=\frac{y}{x}-\xi\frac{y}{x^2}.
$$

\medskip
\noindent
{\bf The supergroup $\OSp(1|2)$.}
The supergroup $\OSp(1|2)$ 
is isomorphic to the supergroup of linear symplectic transformations
of the ${2|1}$-dimensional space equipped with the symplectic form
$$
\om=dp\wedge{}dq+\frac{1}{2}d\t\wedge{}d\t,
$$
where $p,q,\t$ are linear coordinates.

Let $\Rc=\Rc_{\bar0}\oplus\Rc_{\bar1}$ be a commutative ring.
The set of $\Rc$-points of the
supergroup $\OSp(1|2)$  is the following $3|2$-dimensional supergroup of matrices
with entries in $\Rc$:
$$
\left(
\begin{array}{cc|c}
a&b&\g\\[4pt]
c&d&\d\\[4pt]
\hline
\a&\b&e
\end{array}
\right)
\qquad
\hbox{such that}
\qquad
\left\{
\begin{array}{rcl}
ad-bc&=&1-\a\b,\\[4pt]
e&=&1+\a\b,\\[4pt]
-a\d+c\g&=&\a\\[4pt]
-b\d+d\g&=&\b,
\end{array}
\right.
$$
where $a,b,c,d,e\in\Rc_{\bar0}$, and $\a,\b,\g,\d\in\Rc_{\bar1}$.
For properties and applications of this supergroup, see~\cite{Man,GLS}.
Note that the above relations also imply: 
$$
\g=a\b-b\a,
\quad
\d=c\b-d\a,
$$
and $\a\b=\g\d$.

\medskip
\noindent
{\bf Left-invariant vector fields on $\R^{1|1}$ and supersymmetric
linear differential operators.}
Consider the space $\R^{1|1}$ with linear coordinates $(x,\xi)$.
We understand the algebra of algebraic functions on this space as the algebra of polynomials 
in one even and one odd variables:
$$
F(x,\xi)=F_0(x)+\xi{}F_1(x),
$$
where $F_0$ and $F_1$ are usual polynomials in $x$.

The following two vector fields,
characterized by Shander's superversion 
of the rectifiability of vector fields theorem~\cite{Sha}
$$
X=\frac{\partial}{\partial{}x},
\qquad
D=\frac{\partial}{\partial{}\xi}-\xi\frac{\partial}{\partial{}x}
$$
are important in superalgebra and supergeometry.
These vector fields are left-invariant with respect
to the supergroup structure on $\R^{1|1}$ given by
the following multiplication of $\Rc$-points:
$$
(r,\l)\cdot(s,\m)=(r+s+\l\m,\,\l+\m).
$$
Moreover, $X$ and $D$ are characterized
(up to a constant factor) by the property of left-invariance,
as the only even and odd left-invariant vector fields on $\R^{1|1}$,
respectively.

The vector fields $X$ and $D$ form a $1|1$-dimensional Lie superalgebra since 
$$
D^2=\frac{1}{2}[D,D]=-X,
$$
and $[X,D]=0$, with one odd generator $D$.

The space $\R^{1|1}$ equipped with the vector field $D$ is often called by
physicists the $1|1$-dimensional ``superspacetime''.
A {\it supersymmetric} differential operator on $\R^{1|1}$
is an operator that can be expressed as a polynomial in~$D$.

\section*{Appendix~2: Supercontinuants (by Alexey Ustinov)}

This Appendix gives a solution to  Problem~\ref{PbOne}: 
determine the formula for the entries of a superfrieze.

Let $\Rc=\Rc_{\bar0}\oplus\Rc_{\bar1}$ be an arbitrary supercommutative
ring, and the sequences $\{v_i\}$, $\{w_i\}$, with $v_i\in\Rc_{\bar0}$, $w_i\in\Rc_{\bar1}$, be defined by the initial conditions $v_{-1}=0$, $v_0=1$, $w_0=0$ and the recurrence relation
\begin{equation}
\label{1}
v_i=a_iv_{i-1}-v_{i-2}-\beta_iw_{i-1},\quad w_i=w_{i-1}+\beta_iv_{i-1}\quad(i\in \mathbb{Z}).
\end{equation}
In particular,
\begin{gather*}
v_1=a_1,\quad v_2=a_1a_2-1+\beta_1\beta_2,\quad v_3=a_1a_2a_3-a_1-a_3+a_1\beta_2\beta_3+a_3\beta_1\beta_2+\beta_1\beta_3;\\[6pt]
w_1=\beta_1,\quad w_2=a_1\beta_2+\beta_1,\quad w_3= a_1a_2\beta_3+a_1\beta_2+\beta_1\beta_2\beta_3+\beta_1-\beta_3.
\end{gather*}
The problem is to express $v_n$, $w_n$ in terms of $a_1$, \ldots, $a_n$ and $\beta_1$, \ldots, $\beta_n$. Such expression will be called  \textit{supercontinuants}. 
(For the properties of the classical continuants, see~\cite{Knu}.)

We define two sequences of supercontinuants $$\{K\bigl(\begin{smallmatrix} a_1\\
\begin{smallmatrix} \beta_1 & \beta_1
\end{smallmatrix}
\end{smallmatrix}|\begin{smallmatrix}
a_2\\\begin{smallmatrix} \beta_2 & \beta_2
\end{smallmatrix}
\end{smallmatrix}|\ldots |\begin{smallmatrix}
a_n\\\begin{smallmatrix} \beta_n & \beta_n
\end{smallmatrix}
\end{smallmatrix}\bigr)\}\text{\quad and \quad} \{K\bigl(\begin{smallmatrix}  a_1\\\begin{smallmatrix}
\beta_1 & \beta_1
\end{smallmatrix}
\end{smallmatrix}|\ldots |\begin{smallmatrix}
a_{n-1}\\\begin{smallmatrix} \beta_{n-1} & \beta_{n-1}
\end{smallmatrix}
\end{smallmatrix}|\beta_n\bigr)\}$$ by the initial conditions  $K()=1$, $K(\begin{smallmatrix} a_1\\
\begin{smallmatrix} \beta_1 & \beta_1
\end{smallmatrix}
\end{smallmatrix})=a_1$, $K(\beta_1)=\beta_1$  and the recurrence relations
\begin{gather}\nonumber
K\bigl(\begin{smallmatrix} a_1\\[6pt]
	\begin{smallmatrix} \beta_1 & \beta_1
	\end{smallmatrix}
\end{smallmatrix}|\ldots |\begin{smallmatrix}
a_n\\\begin{smallmatrix} \beta_n & \beta_n
\end{smallmatrix}
\end{smallmatrix}\bigr)=a_nK\bigl(\begin{smallmatrix} a_1\\[6pt]
\begin{smallmatrix} \beta_1 & \beta_1
\end{smallmatrix}
\end{smallmatrix}|\ldots |\begin{smallmatrix}
a_{n-1}\\\begin{smallmatrix} \beta_{n-1} & \beta_{n-1}
\end{smallmatrix}
\end{smallmatrix}\bigr)-K\bigl(\begin{smallmatrix} a_1\\
\begin{smallmatrix} \beta_1 & \beta_1
\end{smallmatrix}
\end{smallmatrix}|\ldots |\begin{smallmatrix}
a_{n-2}\\\begin{smallmatrix} \beta_{n-2} & \beta_{n-2}
\end{smallmatrix}
\end{smallmatrix}\bigr)
\\
\label{2}-\beta_nK\bigl(\begin{smallmatrix} a_1\\
	\begin{smallmatrix} \beta_1 & \beta_1
	\end{smallmatrix}
\end{smallmatrix}|\ldots |\begin{smallmatrix}
a_{n-2}\\\begin{smallmatrix} \beta_{n-2} & \beta_{n-2}
\end{smallmatrix}
\end{smallmatrix}|\beta_{n-1}\bigr),\\
\nonumber K\bigl(\begin{smallmatrix}  a_1\\\begin{smallmatrix}
		\beta_1 & \beta_1
	\end{smallmatrix}
\end{smallmatrix}|\ldots |\begin{smallmatrix}
a_{n-1}\\\begin{smallmatrix} \beta_{n-1} & \beta_{n-1}
\end{smallmatrix}
\end{smallmatrix}|\beta_n\bigr)=\beta_nK\bigl(\begin{smallmatrix}  a_1\\\begin{smallmatrix}
	\beta_1 & \beta_1
\end{smallmatrix}
\end{smallmatrix}|\ldots |\begin{smallmatrix}
a_{n-1}\\\begin{smallmatrix} \beta_{n-1} & \beta_{n-1}
\end{smallmatrix}
\end{smallmatrix}\bigr)+K\bigl(\begin{smallmatrix}  a_1\\\begin{smallmatrix}
	\beta_1 & \beta_1
\end{smallmatrix}
\end{smallmatrix}|\ldots |\begin{smallmatrix}
a_{n-2}\\\begin{smallmatrix} \beta_{n-2} & \beta_{n-2}
\end{smallmatrix}
\end{smallmatrix}|\beta_{n-1}\bigr).
\end{gather}
From~\eqref{1} and~\eqref{2} it easily
follows that
\begin{equation*}
v_n=K\bigl(\begin{smallmatrix} a_1\\
	\begin{smallmatrix} \beta_1 & \beta_1
	\end{smallmatrix}
\end{smallmatrix}|\ldots |\begin{smallmatrix}
a_n\\\begin{smallmatrix} \beta_n & \beta_n
\end{smallmatrix}
\end{smallmatrix}\bigr), \qquad
w_n=K\bigl(\begin{smallmatrix}  a_1\\\begin{smallmatrix}
		\beta_1 & \beta_1
	\end{smallmatrix}
\end{smallmatrix}|\ldots |\begin{smallmatrix}
a_{n-1}\\\begin{smallmatrix} \beta_{n-1} & \beta_{n-1}
\end{smallmatrix}
\end{smallmatrix}|\beta_n\bigr).
\end{equation*}

The classical continuants $K(a_1,\ldots,a_n)$,
corresponding to reduced regular continued fractions
$$a_1-\cfrac{1}{a_2-{\atop\ddots\,\displaystyle{-\cfrac{1}{a_n}}}},$$ are defined by
$$K()=1,\quad K(a_1)=a_1,\quad K(a_1,\ldots,a_n)=a_nK(a_1,\ldots,a_{n-1})-K(a_1,\ldots,a_{n-2}).$$
There is Euler's rule which allows one to write down all summands of $K(a_1,\ldots,a_n)$: starting with
the product $a_1 a_2 \ldots a_n$, we strike out adjacent pairs $a_ia_{i+1}$ in all
possible ways. If a pair $a_ia_{i+1}$ is struck out, then it must be replaced by $-1$. 
We can represent Euler's rule graphically by constructing all
``Morse code'' sequences of dots and dashes having length $n$, where each dot
contributes $1$ to the length and each dash contributes $2$. For example $K(a_1,a_2,a_3,a_4)$ consists of the following summands:
\begin{center}
                                        \begin{tikzpicture}[scale=.6]
                                        \draw  (2,0) -- (3,0);\draw  (0,1) -- (1,1);

                                        \filldraw (0,0) circle (0.07);\filldraw (0,1) circle (0.07);\filldraw (0,2) circle (0.07);
                                        \filldraw (1,0) circle (0.07);\filldraw (1,1) circle (0.07);\filldraw (1,2) circle (0.07);
                                        \filldraw (2,0) circle (0.07);\filldraw (2,1) circle (0.07);\filldraw (2,2) circle (0.07);
                                        \filldraw (3,0) circle (0.07);\filldraw (3,1) circle (0.07);\filldraw (3,2) circle (0.07);

                                        \filldraw [right] (5,1) node {$\mapsto -a_3a_4$};
                                        \filldraw [right] (5,0) node {$\mapsto -a_1a_2$};
                                        \filldraw [right] (5,2) node {$\mapsto a_1a_2a_3a_4$};

                                        \begin{scope}[scale=1, xshift=280]

                                        \filldraw (0,1) circle (0.07);\filldraw (0,2) circle (0.07);
                                        \filldraw (1,1) circle (0.07);\filldraw (1,2) circle (0.07);
                                        \filldraw (2,1) circle (0.07);\filldraw (2,2) circle (0.07);
                                        \filldraw (3,1) circle (0.07);\filldraw (3,2) circle (0.07);

                                        \draw  (0,1) -- (1,1);\draw  (2,1) -- (3,1);\draw  (2,2) -- (1,2);

                                        \filldraw [right] (5,1) node {$\mapsto 1$};
                                        \filldraw [right] (5,2) node {$\mapsto -a_1a_4$};

                                        \end{scope}
                                        \end{tikzpicture}
\end{center}

By analogy with  Euler's rule, we can construct a similar rule for calculation of supercontinuants.

\begin{thm}
\label{Thm1U}
The summands of $K\bigl(\begin{smallmatrix} a_1\\
\begin{smallmatrix} \beta_1 & \beta_1
\end{smallmatrix}
\end{smallmatrix}|\begin{smallmatrix}
a_2\\\begin{smallmatrix} \beta_2 & \beta_2
\end{smallmatrix}
\end{smallmatrix}|\ldots \bigr)$ can be obtained from the product $\beta_1\beta_1\beta_2\beta_2\ldots$ by the following rule:  we strike out adjacent pairs and adjacent 4-tuples $\beta_i\beta_i\beta_{i+1}\beta_{i+1}$ in all
possible ways; for deleted  pairs  and 4-tuples we make the substitutions
$\beta_i\beta_i\to a_i$, $\beta_i \beta_{i+1}\to 1$, $\beta_i\beta_i\beta_{i+1}\beta_{i+1}\to -1.$
\end{thm}

This rule can be represented graphically as well. To each monomial there corresponds a sequence of  total length $2n$ (or $2n-1$) consisting of dots (of the length one), dashes (of the length two) and long dashes (of the length four). For example, the monomials of $K\bigl(\begin{smallmatrix} a_1\\
\begin{smallmatrix} \beta_1 & \beta_1
\end{smallmatrix}
\end{smallmatrix}|\begin{smallmatrix}
a_2\\\begin{smallmatrix} \beta_2 & \beta_2
\end{smallmatrix}
\end{smallmatrix}|\beta_{3}\bigr)$ can be obtained from the product $\beta_1\beta_1 \beta_2\beta_2\beta_{3}$ as follows:
\begin{center}
\begin{tikzpicture}[scale=.6]
\draw  (0,1) -- (1,1);\draw  (3,1) -- (4,1);\draw  (0,2) -- (1,2);\draw  (2,2) -- (3,2);

\filldraw (0,0) circle (0.07);\filldraw (0,1) circle (0.07);\filldraw (0,2) circle (0.07);
\filldraw (1,0) circle (0.07);\filldraw (1,1) circle (0.07);\filldraw (1,2) circle (0.07);
\filldraw (2,0) circle (0.07);\filldraw (2,1) circle (0.07);\filldraw (2,2) circle (0.07);
\filldraw (3,0) circle (0.07);\filldraw (3,1) circle (0.07);\filldraw (3,2) circle (0.07);
\filldraw (4,0) circle (0.07);\filldraw (4,1) circle (0.07);\filldraw (4,2) circle (0.07);

\filldraw [right] (5,1) node {$\mapsto a_1\beta_2$}; \draw  (1,0) -- (2,0);    \filldraw [right] (5,0) node {$\mapsto\beta_1\beta_2\beta_3$};
\filldraw [right] (5,2) node {$\mapsto a_1a_2\beta_3$};

\begin{scope}[scale=1, xshift=280]

\filldraw (0,1) circle (0.07);\filldraw (0,2) circle (0.07);
\filldraw (1,1) circle (0.07);\filldraw (1,2) circle (0.07);
\filldraw (2,1) circle (0.07);\filldraw (2,2) circle (0.07);
\filldraw (3,1) circle (0.07);\filldraw (3,2) circle (0.07);
\filldraw (4,1) circle (0.07);\filldraw (4,2) circle (0.07);

\draw  (1,2) -- (2,2);\draw  (3,2) -- (4,2);\draw  (0,1) -- (3,1);

\filldraw [right] (5,1) node {$\mapsto-\beta_3$};
\filldraw [right] (5,2) node {$\mapsto\beta_1$};

\end{scope}
\end{tikzpicture}
\end{center}
Let us note that the odd
variables anticommute with each other. In particular, $\beta_i^2=0$, and in each pair $\beta_i\beta_i$ at least one variable must be struck out. Supercontinuants become the usual continuants if all odd variables are replaced by zeros.

Supercontinuants can be expressed as determinants.

\begin{thm}
\label{Thm2U}
\begin{gather}\label{3}
                                   K\bigl(\begin{smallmatrix} a_1\\
                                   	\begin{smallmatrix} \beta_1 & \beta_1
                                   	\end{smallmatrix}
                                   \end{smallmatrix}|\ldots |\begin{smallmatrix}
                                   a_n\\\begin{smallmatrix} \beta_n & \beta_n
                                   \end{smallmatrix}
                                \end{smallmatrix}\bigr)=\begin{vmatrix}
                                        a_1 & -1+\beta_1\beta_2 & \beta_1\beta_3 & \cdots & \beta_1\beta_{n-1} & \beta_1\beta_n \\
                                        -1 & a_2 & -1+\beta_2\beta_3 & \cdots & \beta_2\beta_{n-1} & \beta_2\beta_n \\
                                        0 & -1 & a_3 & \cdots  & \beta_3\beta_{n-1} & \beta_3\beta_n  \\
                                        \cdots & \cdots  & \cdots & \cdots & \cdots & \cdots \\
                                        0 & \cdots & 0 & -1 & a_{n-1} & -1+\beta_{n-1}\beta_n \\
                                        0 & 0  & \cdots & 0 & -1 & a_n \\
                                        \end{vmatrix},\\[10pt]
                                          \nonumber
                                          K\bigl(\begin{smallmatrix}  a_1\\\begin{smallmatrix}
                                        		\beta_1 & \beta_1
                                        	\end{smallmatrix}
                                        \end{smallmatrix}|\ldots |\begin{smallmatrix}
                                        a_{n-1}\\\begin{smallmatrix} \beta_{n-1} & \beta_{n-1}
                                        \end{smallmatrix}
                                    \end{smallmatrix}|\beta_n\bigr)=\begin{vmatrix}
                                        a_1 & -1+\beta_1\beta_2 & \beta_1\beta_3 & \cdots & \beta_1\beta_{n-1} & \beta_1\\
                                        -1 & a_2 & -1+\beta_2\beta_3 & \cdots & \beta_2\beta_{n-1} & \beta_2 \\
                                        0 & -1 & a_3 & \cdots  & \beta_3\beta_{n-1} & \beta_3\\
                                        \cdots & \cdots  & \cdots & \cdots & \cdots & \cdots \\
                                        0 & \cdots & -1 & a_{n-2} & -1+\beta_{n-2}\beta_{n-1}  & \beta_{n-2}\\
                                        0 & \cdots & 0 & -1 & a_{n-1} & \beta_{n-1}\\
                                        0 & 0  & \cdots & 0 & -1 & \beta_n \\
                                        \end{vmatrix}.
                                        \end{gather}
\end{thm}

The second determinant in  Theorem~\ref{Thm2U} is well-defined
because odd variables occupy only one column. The proofs of
Theorems~\ref{Thm1U} and~\ref{Thm2U} follow by induction from recurrence
relations~\eqref{2}, and we do not dwell on them.

The supercontinuants of the form 
$
K\bigl(\beta_1|\begin{smallmatrix}  a_2\\\begin{smallmatrix}
\beta_2 & \beta_2
\end{smallmatrix}
\end{smallmatrix}|\ldots |\begin{smallmatrix}
a_{n-1}\\\begin{smallmatrix} \beta_{n-1} & \beta_{n-1}
\end{smallmatrix}
\end{smallmatrix}|\beta_n\bigr)
$ also may be defined by the rule from the Theorem~\ref{Thm1U}. For example
$$K(\beta_1|\beta_2)=\beta_1\beta_2+1,\quad K\bigl(\beta_1|\begin{smallmatrix}  a_2\\\begin{smallmatrix}
\beta_2 & \beta_2
\end{smallmatrix}
\end{smallmatrix}|\beta_3\bigr)=a_2\beta_1\beta_3+\beta_1\beta_2+\beta_2\beta_3+1.$$ These supercontinuants can be represented in terms of determinants as well (we assume that the determinant is expanded in the first column, and the same rule is applied to all determinants of  smaller matrices).

\begin{thm}
\label{Thm3U}
The supercontinuants $K\bigl(\beta_1|\begin{smallmatrix}  a_2\\\begin{smallmatrix}
\beta_2 & \beta_2
\end{smallmatrix}
\end{smallmatrix}|\ldots |\begin{smallmatrix}
a_{n-1}\\\begin{smallmatrix} \beta_{n-1} & \beta_{n-1}
\end{smallmatrix}
\end{smallmatrix}|\beta_n\bigr)$ satisfy the
recurrence relation
\begin{equation}
\label{5} 
\begin{array}{rcl}
K\bigl(\beta_1|\begin{smallmatrix}  a_2\\\begin{smallmatrix}
		\beta_2 & \beta_2
	\end{smallmatrix}
\end{smallmatrix}|\ldots |\begin{smallmatrix}
a_{n-1}\\\begin{smallmatrix} \beta_{n-1} & \beta_{n-1}
\end{smallmatrix}
\end{smallmatrix}|\beta_n\bigr)
&=&
-\beta_nK\bigl(\beta_1|\begin{smallmatrix}  a_2\\\begin{smallmatrix}
\beta_2 & \beta_2
\end{smallmatrix}
\end{smallmatrix}|\ldots |\begin{smallmatrix}
a_{n-1}\\\begin{smallmatrix} \beta_{n-1} & \beta_{n-1}
\end{smallmatrix}
\end{smallmatrix}\bigr)\\[6pt]
&&+K\bigl(\beta_1|\begin{smallmatrix}  a_2\\\begin{smallmatrix}
	\beta_2 & \beta_2
\end{smallmatrix}
\end{smallmatrix}|\ldots |\beta_{n-1}\bigr)
\end{array}
\end{equation}
and can be expressed in the following form:
\begin{gather*}
K\bigl(\beta_1|\begin{smallmatrix}  a_2\\\begin{smallmatrix}
		\beta_2 & \beta_2
	\end{smallmatrix}
\end{smallmatrix}|\ldots |\begin{smallmatrix}
a_{n-1}\\\begin{smallmatrix} \beta_{n-1} & \beta_{n-1}
\end{smallmatrix}
\end{smallmatrix}|\beta_n\bigr)=\begin{vmatrix}
                                        \beta_1 & \beta_2 & \beta_3 & \cdots & \beta_{n-1} & 1\\
                                        -1 & a_2 & -1+\beta_2\beta_3 & \cdots & \beta_2\beta_{n-1} & \beta_2 \\
                                        0 & -1 & a_3 & \cdots  & \beta_3\beta_{n-1} & \beta_3\\
                                        \cdots &  \cdots & \cdots & \cdots & \cdots & \cdots \\
                                        0 & \cdots & -1 & a_{n-2} & -1+\beta_{n-2}\beta_{n-1}  & \beta_{n-2}\\
                                        0 & \cdots & 0 & -1 & a_{n-1} & \beta_{n-1}\\
                                        0 & 0  & \cdots & 0 & -1 & \beta_n \\
                                        \end{vmatrix}.
                                        \end{gather*}
\end{thm}

The proof of  formula~\eqref{5} is an application of the
rule from Theorem~\ref{Thm1U}. The determinant formula follows by
induction from the recurrence relation~\eqref{5}.

Finally, the even supercontinuants  $K\bigl(\begin{smallmatrix} a_1\\
\begin{smallmatrix} \beta_1 & \beta_1
\end{smallmatrix}
\end{smallmatrix}|\ldots |\begin{smallmatrix}
a_n\\\begin{smallmatrix} \beta_n & \beta_n
\end{smallmatrix}
\end{smallmatrix}\bigr)$  can be also expressed as Berezinians. Recall that the Berezinian of the matrix
$$
\mathrm{Ber}
\begin{pmatrix}
A&B\\C&D\\
\end{pmatrix},
$$
where $A$ and $D$ have even entries, and $B$ and $C$
have odd entries, is given by (see, e.g.,~\cite{Ber}):
\begin{equation}
\label{Ber} \det(A-BD^{-1}C)\det(D)^{-1}.
\end{equation}

\begin{thm}
\label{Thm4U}
\begin{gather*}
 K\bigl(\begin{smallmatrix} a_1\\
 	\begin{smallmatrix} \beta_1 & \beta_1
 	\end{smallmatrix}
 \end{smallmatrix}|\ldots |\begin{smallmatrix}
 a_n\\\begin{smallmatrix} \beta_n & \beta_n
 \end{smallmatrix}
\end{smallmatrix}\bigr)=\mathrm{Ber}\begin{pmatrix}
A&B\\C&D\\                              \end{pmatrix},
\end{gather*}
where
\begin{gather*}
A=\begin{pmatrix}
a_1&-1&0&\cdots&0\\-1&a_2&-1&\ddots&\vdots\\
0&-1&a_3&\ddots&0\\
\vdots&\ddots&\ddots&\ddots&-1\\
0&\cdots&0&-1&a_n\\
\end{pmatrix},\qquad B=\begin{pmatrix}
\beta_1&\beta_2 &\beta_3&\cdots&\beta_n \\0&\beta_2&\beta_3&\ddots&\vdots\\
0&0&\beta_3&\ddots&\beta_n\\
\vdots&\ddots&\ddots&\ddots&\beta_n\\
0&\cdots&0&0&\beta_n\\
\end{pmatrix},\\[10pt]
C=\begin{pmatrix}
-\beta_1&0&0&\cdots&0\\0&-\beta_2&0&\ddots&\vdots\\
0&0&-\beta_3&\ddots&0\\
\vdots&\ddots&\ddots&\ddots&0\\
0&\cdots&0&0&-\beta_n\\
\end{pmatrix},\qquad D=\begin{pmatrix}
1&0 &0&\cdots&0 \\0&1&0&\ddots&\vdots\\
0&0&1&\ddots&0\\
\vdots&\ddots&\ddots&\ddots&0\\
0&\cdots&0&0&1\\
\end{pmatrix}.
\end{gather*}
\end{thm}

Theorem~\ref{Thm4U} is direct corollary of~\eqref{5}
and~\eqref{Ber}.

It follows from recurrence relations~\eqref{2} and~\eqref{5} that  the
number of terms in supercontinuants $ K\bigl(\begin{smallmatrix} a_1\\
\begin{smallmatrix} \beta_1 & \beta_1
\end{smallmatrix}
\end{smallmatrix}|\ldots |\begin{smallmatrix}
a_n\\\begin{smallmatrix} \beta_n & \beta_n
\end{smallmatrix}
\end{smallmatrix}\bigr)$, $K\bigl(\begin{smallmatrix}  a_1\\\begin{smallmatrix}
\beta_1 & \beta_1
\end{smallmatrix}
\end{smallmatrix}|\ldots |\begin{smallmatrix}
a_{n-1}\\\begin{smallmatrix} \beta_{n-1} & \beta_{n-1}
\end{smallmatrix}
\end{smallmatrix}|\beta_n\bigr)$ and $K\bigl(\beta_1|\begin{smallmatrix}  a_2\\\begin{smallmatrix}
\beta_2 & \beta_2
\end{smallmatrix}
\end{smallmatrix}|\ldots |\begin{smallmatrix}
a_{n-1}\\\begin{smallmatrix} \beta_{n-1} & \beta_{n-1}
\end{smallmatrix}
\end{smallmatrix}|\beta_n\bigr)$
coincide respectively with the sequences (see~\cite{OEIS})
\begin{align*}
A077998&:1, 3, 6, 14, 31, 70, 157, 353, 793, 1782, 4004,  \ldots\\
A006054&:1, 2, 5, 11, 25, 56, 126, 283, 636, 1429, 3211, \ldots\\
A052534&:1, 2, 4,  9, 20, 45, 101, 227, 510, 1146, 2575,
\ldots
\end{align*}

\vskip 0.5cm
{\bf Acknowledgements}.
The first three authors would like to thank the Centro Internazionale per la Ricerca Matematica, 
the Mathematics Department of the University of Trento
and the foundation Bruno Kessler for excellent conditions they offered us. 
We are pleased to thank Frederic Chapoton and Dimitry Leites
 for interesting discussions, special thanks to Dimitry for a careful reading of the first version of this paper. 
S.~M-G. and V.~O. are grateful to the Institute for Computational and Experimental Research in Mathematics
for its hospitality.
S.~M-G. and V.~O. were partially supported by the PICS05974 ``PENTAFRIZ'' of CNRS. 
S.~T. was  supported by   NSF grant DMS-1105442. 
A.~U.'s research was supported by Russian Science Foundation (Project N~14-11-00335).

\end{document}